\documentclass[draft,reqno,10pt, centertags]{amsart}
\usepackage{amsmath,amsthm,amscd,amssymb}
\usepackage{upref}
\usepackage{latexsym}
\usepackage{lscape}
\usepackage{graphicx}
\usepackage{multirow}

\makeatletter
\def\theequation{\@arabic\c@equation}

\newcommand{\gaD}{\gamma_{{}_D}}

\newcommand{\gaN}{\gamma_{{}_N}}
\newcommand{\tN}{\tau_{{}_N}}
\newcommand{\Om}{\Omega}
\newcommand{\dOm}{{\partial\Omega}}
\newcommand{\lnth}{_{N^{3/2}(\partial\Omega)}\langle}
\newcommand{\rnths}{\rangle_{(N^{3/2}(\partial\Omega))^*}}

\newcommand{\e}{\hbox{\rm e}}

\newcommand{\bbM}{{\mathbb{M}}}
\newcommand{\bbN}{{\mathbb{N}}}

\newcommand{\R}{{\mathbb{R}}}

\newcommand{\bbZ}{{\mathbb{Z}}}
\newcommand{\Z}{{\mathbb{Z}}}
\newcommand{\C}{{\mathbb{C}}}

\newcommand{\bbC}{{\mathbb{C}}}

\newcommand{\cB}{{\mathcal B}}
\newcommand{\cC}{{\mathcal C}}
\newcommand{\cD}{{\mathcal D}}

\newcommand{\cG}{{\mathcal G}}
\newcommand{\cH}{{\mathcal H}}

\newcommand{\cJ}{{\mathcal J}}
\newcommand{\cK}{{\mathcal K}}

\newcommand{\cM}{{\mathcal M}}
\newcommand{\cN}{{\mathcal N}}

\newcommand{\cQ}{{\mathcal Q}}

\newcommand{\cU}{{\mathcal U}}

\newcommand{\cX}{{\mathcal X}}
\newcommand{\cY}{{\mathcal Y}}
\newcommand{\bfi}{{\bf i}}

\newcommand{\no}{\nonumber}
\newcommand{\lb}{\label}

\newcommand{\ol}{\overline}

\newcommand{\ran}{\text{\rm{ran}}}

\newcommand{\bi}{\bibitem}

\newcommand{\hatt}{\widehat}

\newcommand{\dist}{\operatorname{dist}}

\newcommand{\mi}{\operatorname{Mas}}
\newcommand{\mo}{\operatorname{Mor}}

\newcommand{\noh}{N^{1/2}(\partial\Omega)}
\newcommand{\nohs}{\big(N^{1/2}(\partial\Omega)\big)^*}
\newcommand{\rnohs}{\rangle_{(N^{1/2}(\partial\Omega))^*}}
\newcommand{\lnoh}{_{N^{1/2}(\partial\Omega)}\langle}

\newcommand{\gd}{\widehat{\gamma}_D}
\newcommand{\gn}{\widehat{\gamma}_N}
\newcommand{\Ltwo}{L^2(Q, d^nx)}

\newcommand{\Ltwom}{L^2(Q;\C^m)}
\newcommand{\Htworm}{H^2(Q;\R^{2m})}
\newcommand{\Htwom}{H^2(Q;\C^m)}

\newcommand{\Ltworm}{L^2(Q;\R^{2m})}

\newcommand{\Ltwor}{L^2(Q;\mathbb{R})}

\newcommand{\Ltwotm}{L^2(tQ;\mathbb{C}^m)}
\newcommand{\Htwotm}{H^2(tQ;\mathbb{C}^m)}

\newcommand{\Htwortm}{H^2(tQ;\mathbb{R}^{2m})}
\newcommand{\Ltwortm}{L^2(tQ;\mathbb{R}^{2m})}

\newcommand{\Htwor}{H^2(Q;\mathbb{R})}

\newcommand{\Htworo}{H^2_{0}(Q;\mathbb{R})}

\numberwithin{equation}{section}

\renewcommand{\det}{\operatorname{det}}

\newcommand{\dom}{\operatorname{dom}}
\newcommand{\codim}{\operatorname{codim}}

\newcommand{\tr}{\operatorname{tr}}

\newcommand{\sign}{\operatorname{sign}}

\newcommand{\spec}{\operatorname{Spec}}

\renewcommand{\Re}{\operatorname{Re }}
\renewcommand\Im{\operatorname{Im}}
\renewcommand{\ker}{\operatorname{ker}}

\theoremstyle{plain}
\newtheorem{theorem}{Theorem}[section]
\newtheorem{hypothesis}[theorem]{Hypothesis}
\newtheorem{lemma}[theorem]{Lemma}
\newtheorem{corollary}[theorem]{Corollary}
\newtheorem{proposition}[theorem]{Proposition}

\theoremstyle{definition}
\newtheorem{definition}[theorem]{Definition}

\newtheorem{remark}[theorem]{Remark}

\begin{document}

\begin{abstract}
We study the spectrum of Schr\"odinger operators with matrix valued potentials utilizing tools from infinite dimensional symplectic geometry. Using the spaces of abstract boundary values, we derive relations between the Morse and Maslov indices for a family of operators on a Hilbert space obtained by perturbing a given self-adjoint operator by a smooth family of bounded self-adjoint operators. The abstract results are applied to the Schr\"{o}dinger operators with quasi-periodic, Dirichlet and Neumann boundary conditions. In particular, we derive an analogue of the Morse-Smale Index Theorem for the multidimensional 
Schr\"odinger operators with periodic potentials. For quasi convex domains in $\R^n$ we recast the results connecting the Morse and Maslov indices using the Dirichlet and Neumann traces on the boundary of the domain.
\end{abstract}

\allowdisplaybreaks

\title[Periodic Schr\"odinger Operators]{The Morse and Maslov Indices for Schr\"odinger Operators}

\author[Y. Latushkin]{Yuri Latushkin}
\address{Department of Mathematics,
The University of Missouri, Columbia, MO 65211, USA}
\email{latushkiny@missouri.edu}
\author[A. Sukhtayev]{Alim Sukhtayev}
\address{Department of Mathematics, Texas A\&M University \\
College Station, TX 77843-3368, USA}
\email{alim@math.tamu.edu}
\author[S. Sukhtaiev]{Selim Sukhtaiev	}
\address{Department of Mathematics,
The University of Missouri, Columbia, MO 65211, USA}
\email{sswfd@mail.missouri.edu}
\date{\today}
\keywords{Schr\"odinger equation, Hamiltonian systems, periodic potentials,  stability, differential operators, discrete spectrum, Fredholm Lagrangian Grassmanian, symplectic}
\thanks{Partially supported by the NSF grant  DMS-1067929, and by the Research Council and the Research
Board of the University of Missouri. Yuri Latushkin sincerely thanks Lai-Sang Young for the opportunity to spend his sabbatical at the Courant Institute where this paper was completed. The authors thank Margaret Beck, Graham Cox and Chris  Jones for many stimulating discussions.}
\maketitle

\section{Introduction}

The classical Morse Index Theorem relates the Morse index (the number of negative eigenvalues counting the multiplicities)  of the
one-dimensional Schr\"odinger operator $L=-\partial^2_{D,[0,1]}+V$ in $L^2([0,1])$ with a real valued potential $V$ and the Dirichlet boundary conditions, and the total dimension of the null spaces of the restrictions $L_t=-\partial^2_{D,[0,t]}+V$ of $L$ in $L^2([0,t])$,
 see, e.g., \cite[Section 15]{M63}:
\begin{equation}\label{MT}
\mo(L)=\sum_{t\in[0,1)}\dim(\ker(L_t)).
\end{equation}
This foundational result was developed and generalized in many important directions, in particular, for the case of periodic boundary conditions \cite{B56,CZ84,CD77,D76}. 
A point $t\in[0,1]$ is called conjugate, or a crossing, if the
null space of the restriction $L_t=-\partial^2_{D,[0,t]}+V$ of $L$ in $L^2([0,t])$ is nontrivial. The right-hand side of \eqref{MT} therefore can be viewed as the Maslov index (the number of the conjugate points
counting the multiplicities) of a certain path in the space of Lagrangian planes (among many important contributions on the Arnold-Keller-Malsov index  we mention \cite{A01,arnold67,Arn85,CLM,G,F04,rs93,RoSa95}).

The celebrated  multidimensional Morse-Smale Index Theorem \cite{S65} gives  relation \eqref{MT} between the Morse and Maslov indices for elliptic differential operators $L$ on a multidimensional manifold with the Dirichlet boundary conditions and their restrictions $L_t$ on a family of submanifolds. This result is related to many  important  advances in the study of the Maslov index, and we refer to \cite{BW93,DJ11,Sw,SW,U73} and the bibliography therein. While the Dirichlet and some other separated boundary conditions have been studied in detail, the multidimensional periodic case is not well understood. Although there is a tremendous literature on the spectrum of the periodic Schr\"odinger operator (see, e.g., the bibliography in \cite{Kar}), so far a link to its topologic description 
via the Maslov index appears to be missing, and is supplied in this paper.

The main objective of the current paper is to prove a version of the Morse-Smale Index Theorem for the multidimensional differential operators with periodic and quasi-periodic boundary conditions.
Recently, relations between the Morse and Maslov indices have been extensively studied in \cite{DJ11} and \cite{JLM,CJLS,CJM}, see also 
\cite{CB,CDB06,CDB09,CDB11}. 
A commonly used method for elliptic operators on multidimensional domains $\Omega\subset\R^n$ is to construct a differentiable loop of Lagrangian planes in the symplectic Hilbert space $H^{1/2}(\partial \Omega)\times H^{-1/2}(\partial \Omega)$ and examine the intersections of this loop with a fixed plane, which is determined by the type of boundary conditions under consideration. The path in the space of Lagrangian planes is constructed by taking boundary traces of the weak solutions of the respective homogeneous equations. 

In the current paper we take a different path. We do not use the weak solutions but instead, using a functional analysis approach  to the Maslov index developed in \cite{BbF95}, we consider yet another symplectic Hilbert space, the space of abstract boundary values. This helps us to avoid taking the Dirichlet and Neumann traces of $H^1(\Omega)$ functions, and therefore eleminates many of the technical difficulties encountered in \cite{CJLS,CJM, DJ11}. In particular, we do not need to show that the loop in the space of Lagrangian planes is 
differentiable as this fact comes for free from the abstract theory in \cite{BbF95}.

Let $A$ be a symmetric operator in a Hilbert space $\cH$.
 The quotient space $\dom(A^*)/\dom(A)$ has a natural symplectic structure; it is called in \cite{BbF95} the space of abstract boundary values, and we let $\gamma$ denote the natural projection from $\dom(A^*)$ into this space.
Let $A_\cD$ be a self-adjoint extension of $A$ with the domain $\cD\subset \cH$. We assume that $A_\cD $ has compact resolvent. Let $V_t:[0,1]\rightarrow \cB(\cH)$ be a differentiable family of bounded self-adjoint operators satisfying certain hypotheses listed below in Theorem \ref{mormas} {\it (i)}. The main abstract result of the current  paper is concerned with counting the number of the unstable eigenvalues of $A_\cD+V_t$, that is, the Morse index of this operator, 
$\mo(A_\cD+V_t):=\#\left\{ \lambda_k<0 \big| \lambda_k\in\spec(A_\cD+V_t) \right\}$. We recast the existence of a non-trivial solution of the eigenvalue problem 
\begin{align}
&A^*u+V_tu=\lambda u,\ \ u\in \cD,\lb{evp1}
\end{align}
in terms of the conjugate points (or crossings), i.e., the points $t\in[0,1]$ where an associated path of Lagrangian subspaces in the space of abstract boundary values intersects the train of $\gamma(\cD)$ (the set of the subspaces in $\dom(A^*)/\dom(A)$ with a nontrivial intersection with $\gamma(\cD)$).

Let $\Gamma=\cup_{j=1}^4{\Gamma}_j$ be the boundary of the square $[\lambda_{\infty},0]\times [\tau,1]$ for some (small) $\tau>0$ and $\lambda_{\infty}<0$ (with large $|\lambda_{\infty}|$), parametrized by a segment $\Sigma\subset\R$ as described below in \eqref{par2}-\eqref{par5}, see Figure 1. Denoting $W_s:=V_{t(s)}-\lambda(s)I_{\cH}$ for $s\in\Sigma$, we notice that  problem \eqref{evp1} with $\lambda=\lambda(s)$ and $t=t(s)$ has a nontrivial solution if and only if $\gamma\left(\ker\left(A^*+ W_s\right)\right)\cap\gamma\left(\cD\right)\not=\{0\}$. The points $s\in\Sigma$ or $t=t(s)$ where this happens are called conjugate points; they are shown in Figure 1 as black circles. Each conjugate point is assigned certain signature (or multiplicity)  equals to the signature of a finite dimensional (Maslov) crossing form, see Definition \ref{def21} below. The Maslov index $\mi(\Upsilon(s)\big|_{s\in\Sigma}, \gamma(\cD))$ of the path $\Upsilon:\Sigma\mapsto\gamma(\ker(A^*+W_s))$ is defined as the number of the conjugate points counting their signatures, see definition \eqref{dfnMInd} below. Using homotopy invariance of the Maslov index we conclude that the total number of the conjugate points on $\Gamma$ counting their signatures is equal to zero. Using this and analyzing sign definiteness of the crossing forms, we prove the following general formulas:
\begin{align}\lb{mff1intro}
\mo(A_\cD+V_{\tau})-\mo(A_\cD+V_{1})&=\mi(\gamma(\ker(A^*)+V_t)\big|_{\tau\leq t\leq 1},\gamma(\cD)),\\
\lb{mff2intro}
\mo(A_\cD+V_{\tau})-\mo(A_\cD+V_{1})&=\sum\limits_{\tau< t \leq 1}\dim\big(\ker(A_\cD+V_t)\big),\\ 
\mo(A_\cD+V_{\tau})-\mo(A_\cD+V_{1})&=-\sum\limits_{\tau\leq t < 1}\dim\big(\ker(A_\cD+V_t)\big),\lb{mff3intro}
\end{align}
where formula \eqref{mff2intro}, respectively, \eqref{mff3intro} holds provided $\frac {dV_t}{dt}$ is positive, respectively, negative definite).
\begin{figure}
	\begin{picture}(100,100)(-20,0)
	\put(78,0){0}
	\put(80,8){\vector(0,1){95}}
	\put(5,10){\vector(1,0){95}}
	\put(71,40){\text{\tiny $\Gamma_2$}}
	\put(12,40){\text{\tiny $\Gamma_4$}}
	\put(45,75){\text{\tiny $\Gamma_3$}}
	\put(43,14){\text{\tiny $\Gamma_1$}}
	\put(100,12){$\lambda$}
	\put(85,100){$s$}
	\put(80,20){\line(0,1){60}}
	\put(10,20){\line(0,1){60}}
	\put(80,8){\line(0,1){4}}
	\put(2,2){$\lambda_\infty$}
	\put(10,20){\line(1,0){70}}
	\put(10,80){\line(1,0){70}}
	\put(65,20){\circle*{4}}
	\put(80,50){\circle*{4}}
	\put(80,70){\circle*{4}}
	\put(20,80){\circle*{4}}
	\put(40,80){\circle*{4}}
	\put(60,80){\circle*{4}}
	\put(20,87){{\tiny \text{eigenvalues}}}
	\put(14,24){{\tiny \,\,\,\text{eigenvalues}}}  
	\put(85,23){\rotatebox{90}{{\tiny conjugate points}}}
	\put(82,18){{\tiny $\tau$}}
	\put(82,80){{\tiny $1$}}
	\put(0,25){\rotatebox{90}{{\tiny no intersections}}}
	\end{picture}
	\caption{\ }
\end{figure} 

Our main application of the general formulas \eqref{mff1intro}--\eqref{mff3intro} is a
relation between the Morse and Maslov indices for the
multidimensional Schr\"{o}dinger operator with periodic matrix valued potential and $\vec{\theta}$-periodic boundary conditions (the one dimensional case was treated in \cite{JLM} by a different method). Let $V\in C^1(\R^n,\R^{m\times m} )$ be a periodic function with the basic period cell $$Q:=\{t_1{a}_1+\cdots+t_n{a}_n|\ 0\leq t_j\leq 1 , j\in\{1,\dots, n\}\},$$ and assume that $V(x)=V(x)^{\top},\ x\in Q$. For a given vector $\vec{\theta}\in[0,1)^n$, we consider the Laplace operator $-\Delta_{\vec{\theta}}$ in $L^2(Q, \C^m)$ equipped with the $\vec{\theta}$-periodic boundary conditions 
$$u(x+a_j)=\e^{2\pi\bfi\theta_j}u(x),\ \frac{\partial u}{\partial \vec{\nu}}(x+a_j)=\e^{2\pi\bfi\theta_j}\frac{\partial u}{\partial \vec{\nu}}(x)\text{ for a.\ a.\ $x\in {\partial Q}^0_{j},  j\in\{1,\dots, n\}$,}$$
where ${\partial Q}^0_{j}=\big\{t_1a_1+\dots+t_na_n\in Q\big| t_j=0\big\}$ is the ``left'' $j$-th face of $Q$ (the rigorous definition of
$-\Delta_{\vec{\theta}}$  is given in Theorem \ref{operdef1} below). 
Shrinking $x\mapsto y=tx$ for $x\in Q$  produces the set
$tQ=\{y=tx\big| x\in Q\}$ for each $t\in(0,1]$. Considering the operator $-\Delta_{\vec{\theta},t}+V_{tQ}(y)$ in $L^2(tQ; \C^m)$ with $V_{tQ}=V\big|_{tQ}$ and rescaling it back to
$Q$, we obtain a family of operators  $-\Delta_{\vec{\theta}}+t^2V(tx)$ in $L^2(Q; \C^m)$. Applying the abstract result in \eqref{mff1intro} with $A_\cD=-\Delta_{\vec{\theta}}$ and $V_t=t^2V(tx)$, we arrive at a formula relating the Morse index of the operator $-\Delta_{\vec{\theta}}+V$ and the Maslov index of a flow the set of Lagrangian subspaces 
in the space of abstract boundary values, see Theorems \ref{mormastper} 
and \ref{mormasper}. In addition, we give a fairly explicit formulas for the Maslov crossing forms in term of the potential, see Theorems \ref{3.19teor} and \ref{onedim}.
In turn, the relation between the Morse and Maslov indices implies the following versions of the Morse-Smale Index Theorem:
{\em If $\min\spec\big(2tV(tx)+t^2\nabla V(tx)x\big)>0$ for each $t~\in~(0,1]$ and almost all $x\in Q$ then 
	\begin{align}
	&\mo\big(-\Delta_{\vec{\theta}}+V\big)\no\\	
	&=\begin{cases} 0, &\text{if $\vec{\theta}\not=0$}, \no\\
	\mo\big(V(0)\big)-\sum\limits_{\tau< t \leq 1}\dim\big(\ker(-\Delta_{0,t}+V_{tQ})\big),& \text{if $\vec{\theta}=0$, $V(0)$ is invertible,}
	\end{cases} \no
	\end{align}
if $\min\spec\big(2tV(tx)+t^2\nabla V(tx)x\big)<0$ then
	\begin{align}
	&\mo\big(-\Delta_{\vec{\theta}}+V\big)\no\\	
	&=\begin{cases} \sum\limits_{\tau\leq t < 1}\dim\big(\ker(-\Delta_{\vec{\theta},t}+V_{tQ})\big), &\text{if $\vec{\theta}\not=0$}, \no\\
	\mo\big(V(0)\big)+\sum\limits_{\tau\leq t < 1}\dim\big(\ker(-\Delta_{0,t}+V_{tQ})\big), & \text{if $\vec{\theta}=0$, $V(0)$ is invertible.}
	\end{cases} \no
	\end{align}
	}
We also apply the abstract formulas  \eqref{mff1intro}--\eqref{mff3intro} to derive analogous results for the Schr\"{o}dinger operator in $L^2(\Omega)$ for a domain $\Omega\subset\R^n$ with Dirichlet and Neumann boundary conditions. This derivation is substantially easier than in \cite{DJ11} or \cite{CJLS}, although we were not able to achieve the generality of \cite{CJLS} since for the more general Neumann type boundary conditions considered in \cite{CJLS} one needs the spaces of abstract boundary values
with $t$-dependent $\dom(A)$, cf.\ \cite{SW}. 

Up to this point in our discussion the Morse index of the differential operators on $\Omega\subset\R^n$ was computed via the Maslov index of a path of Lagrangian subspaces in the symplectic space of abstract boundary values. 
The next natural question is how to transfer the information from the space of abstract boundary values into the boundary space of functions (or distributions) on the actual boundary $\partial\Omega$ of the domain.
In the current paper we answer this question for a quite general class of quasi-convex domains. We choose
$(N^{1/2}(\dOm))^*\times N^{1/2}(\dOm)$ as the boundary space; here
$N^{1/2}(\dOm)$ is defined in \eqref{dfnN12} and satisfies $N^{1/2}(\dOm)=H^{1/2}(\dOm)$ when $\dOm$ is sufficiently smooth. We construct a symplectomorphism between the space of abstract boundary values and the boundary space. Using the abstract formulas \eqref{mff1intro}--\eqref{mff3intro} we characterize the Morse index of the Shr\"odinger operator via the Maslov index of a path of Lagrangian planes in the boundary space $(N^{1/2}(\dOm))^*\times N^{1/2}(\dOm)$. Although in spirit this result is close to the similar formulas from \cite{CJLS}, it is not the same as in the current paper we use different from \cite{CJLS} trace operators.   

{\bf Notations.} We denote by $I_n$ and 
$0_n$ the $n\times n$ identity and zero matrix.  We let $\top$ denote transposition. For an $n\times m$ matrix $A=(a_{ij})_{i=1,j=1}^{n,m}$
and a $k\times\ell$ matrix $B=(b_{ij})_{i=1,j=1}^{k,\ell}$, we denote by
$A\otimes B$ the Kronecker product, that is, the $nk\times m\ell$ matrix composed of $k\times\ell$ blocks $a_{ij}B$, $i=1,\dots n$, $j=1,\dots m$. We let $\cJ=I_m\otimes\left[\begin{smallmatrix}0&-1\\1&0\end{smallmatrix}\right]$.
For a given function $V: \R^n\rightarrow \R^{m\times m}$  we denote $$V_t(x):=t^2V(tx),\, V_{\R}(x):=V(x)\otimes I_2,\, V_{t,\R}(x):=t^2V(tx)\otimes I_2, \, x\in\R^n,\, t\in\R.$$ We let $(\cdot\,,\cdot)_{\cX}$ denote the scalar product in a Hilbert space $\cX$. The Banach spaces of bounded and compact linear operators on $\cX$ are
denoted by $\cB(\cX)$ and $\cB_\infty(\cX)$, respectively. The spectrum of an operator $A$ is denoted by $\spec(A)$. In addition, $\#(M)$ abbreviates the cardinality of the set $M$. 
 We denote by $\Ltwor$ the (real) space of the real-valued functions  and by $\Ltwo=L^2(Q;\C)$ the (complex) space of the complex-valued $L^2$-functions. Similarly, we use notations $\Htwor$ and $\Htworo$ for the spaces of the real-valued functions from the corresponding Sobolev spaces, while $H^2(Q;\R^m)$ and $H^2_0(Q;\R^m)$ denote the respective spaces of the vector valued functions.
\section{Abstract results}\label{sec:2}
\subsection{Definition of the Maslov index.}
In this section we derive our main abstract result relating the Morse and Maslov indices of a family of self-adjoint operators. To begin, we recall basic definitions, see \cite{BbF95, CJLS,F04} for more details. Let $\cH $ be a real Hilbert space with the inner product $\left( \cdot, \cdot \right) _{\cH}$ and $\omega_{\cH}:\cH\times\cH\rightarrow \R$ be a bilinear, skew-symmetric, bounded, non-degenerate form. There exists an operator $J\in \cB(\cH),\ J^2=-I_{\cH},\ J^{t}=-J$, such that $\omega_{\cH}(u,v)=(Ju,v)_{\cH}$ for all $u,v\in\cH$, where $J^t$ is the operator adjoint to $J$ in $\left(\cH, \left( \cdot, \cdot \right) _{\cH}\right)$. For a given subspace $\cX\subset\cH$ we denote $\cX^0:=\left\{u\in \cH\big|\ \omega_{\cH}(u,v)=0 \text{\ for all\ } v\in\cH \right\}$ and say that $\cX$ is isotropic if $\cX\subset \cX^0$, coisotropic if $\cX^0\subset \cX$ and Lagrangian if $\cX=\cX^0$. The Lagrangian-Grassmannian is the set $\Lambda(\cH)$ of the Lagrangian subspaces of $\cH$, equipped with the metric $d(\cX,
 \cY):=\|P_{\cX}-P_{\cY}\|_{\cB(\cH)}$, where $P_{\cX}$ is the orthogonal projection onto $\cX$.  A pair of subspaces $(\cX,\cY)$ is called Fredholm if $\dim(\cX\cap \cY)<+\infty$ and $\cX+\cY$ is closed and $\codim(\cX+\cY)<+\infty$. For a fixed Lagrangian subspace $\cX$ we denote $F\Lambda (\cX):=\left\{\cY\in \Lambda(\cX) \big| \left(\cX, \cY \right) \text{\ is a Fredholm pair}\right\}$, the set $F\Lambda(\cX)$, equipped with metric $d$ is called the Fredholm-Lagrangian-Grassmannian of $\cX$. 

For any Lagrangian subspace $\cX$, one has $(\cX)^{\perp}=J\cX$ and thus for each $u\in\cH$ there exist $u_1, u_2\in \cX$, such that $u=u_1+Ju_2$. For a given complex number $\alpha+\bfi\beta$ we define $(\alpha+\bfi\beta)u:=\alpha u_1-\beta u_2+J(\alpha u_2+\beta u_1)$. We equip $\cH$ with this multiplication by complex numbers and the standard addition and denote the obtained this way vector space by $\cH_J$. Note, that $\cH_J$ does not depend on the choice of the Lagrangian subspace $\cX$ \cite{F04}. The vector space $\cH_J$ becomes a Hilbert space when equipped with the inner product $(u,v)_J:=(u,v)_{\cH}-\bfi\omega_{\cH}(u,v)$, $u,v\in \cH_J$.   

For a given Lagrangian subspace $\cX$, we define the Souriau map as 
\begin{align}
& S_{\cX}:F\Lambda(\cX)\rightarrow \cB(\cH_J),\,
S_{\cX}(\cY):= (I_{\cH_{J}}-2P_\cY)(2P_\cX-I_{\cH_{J}}),
\end{align}
and recall from \cite[Propositions 2.46, 2.52]{F04} that 
\begin{equation}\no
S_{\cX}(\cY)\in \cU(\cH_J) \text{ is a unitary operator and $S_{\cX}(\cY)+I_{\cH_J}$ is Fredholm operator,\ } 
\end{equation}
\begin{equation}\no
\dim_{\R}(\cX \cap \cY)= \dim_{\C}\ker(S_{\cX}(\cY)+I_{\cH_J}).
\end{equation}
We consider a continuous path $\Upsilon$ with values in $F\Lambda(\cX)$, i.e. $\Upsilon\in C([a,b], F\Lambda(\cX))$, and the corresponding path $\upsilon:t\mapsto S_{\cX}(\Upsilon(t))$ in $\cU(\cH_J)$.  The Maslov index of $\Upsilon$ is defined as the spectral flow of the family $\upsilon$ through the point $-1$ which could be only an isolated point in $\spec(\upsilon(t))$ since the operators $\upsilon(t)+I_{\cB(\cH_J)}$ are Fredholm and the spectrum of the unitary operators $\upsilon(t)$ belong to the unite circle. It follows that there exists a partition $a=t_0<t_1<\cdots<t_N=b$ and positive numbers $\varepsilon_j\in(0,\pi)$, such that  $\e^{\bfi(\pi+\varepsilon_j)}\not \in \spec (\upsilon(t))$ for each $1\leq j\leq N$. We define 
$k(t,\varepsilon):=\sum\limits_{0\leq \theta\leq \varepsilon}\ker(\upsilon(t)-\e^{\bfi(\pi+\theta)})$ and  introduce the Maslov index as follows:
\begin{equation}\lb{dfnMInd}
\text{Mas}(\Upsilon,\cX):=\sum\limits_{j=1}^{N}\left(k(t_j,\varepsilon_j)-k(t_{j-1},\varepsilon_j)\right),
\end{equation}
see \cite[Definition 3.2]{F04}. By \cite[Proposition 3.3]{F04} the number Mas$(\Upsilon,\cX)$ is well defined, that is, it is independent on the choice of the partition $t_j$ and $\varepsilon_j$. 

We will now recall how to compute the Maslov index via crossing forms. Assume that $\Upsilon\in C^1([a,b], F\Lambda(\cX))$ and let $t_*\in[a,b]$. By \cite[Lemma 3.8]{CJLS} there exists a neighbourhood $\Sigma_0$ of $t_*$ and a $C^1(\Sigma_0, \cB(\Upsilon(t_*), \Upsilon(t_*)^{\perp}))$ family $R_t$, such that $\Upsilon(t)=\{u+R_tu\big| u\in \Upsilon(t_*)\}$, for $t\in \Sigma_0$. We will use the following terminology from \cite[Definition 3.20]{F04}.
\begin{definition}\label{def21} Let $\cX$ be a Lagrangian subspace and $\Upsilon\in C^1([a,b], F\Lambda(\cX))$.

{\it (i)} We call $t_*\in[a,b]$ a conjugate point or crossing if $\Upsilon(t_*)\cap \cX\not=\{0\}$.

{\it (ii)} The finite dimentional form $$\cQ_{t_*,\cX}(u,v):=\frac{d}{dt}\omega_{\cH}(u,R_tv)\big|_{t=t_*}=\omega_{\cH}(u, \dot{R}_{t=t_*}v), \text{\ for\ }u,v \in \Upsilon(t_*)\cap \cX,$$  is called the crossing form at the crossing $t_*$.

{\it (iii)} The crossing $t_*$ is called regular if the form $\cQ_{t_*,\cX}$ is non-degenerate, positive if $\cQ_{t_*,\cX}$ is positive definite, and negative if $\cQ_{t_*,\cX}$ is negative definite.

\begin{theorem} [\cite{F04}, Corollary 3.25]\lb{masform}
If $t_*$ is a regular crossing of a $C^1([0,1], F\Lambda(\cX))$ path $\Upsilon$ then there exists $\delta>0$ such that

(i) $\mi(\Upsilon_{|t-t_*|<\delta},\cX)=\sign \cQ_{t_*,\cX}$, if $t_*\in (0,1)$,

(ii) $\mi(\Upsilon_{0\leq t\leq \delta},\cX)=-n_{-}(\cQ_{t_*,\cX})$, if $t_*=0$,

(iii) $\mi(\Upsilon_{1-\delta\leq t\leq 1}, \cX)=n_{+}(\cQ_{t_*,\cX})$, if $t_*=1$.
\end{theorem}
\end{definition}
\subsection{The symplectic space of abstract boundary values.}
Let $A$ be a densely defined, closed, symmetric operator in $\cH$, i.e. $A:\dom(A)\subset \cH\rightarrow \cH$, and  $(Ax,y)_{\cH}=(x,Ay)_{\cH}\text {\ for all\ } x,y \in \dom(A)$, so that $A\subset A^*.$ We denote $\cD_{\max}:=\dom(A^*)$ and $\cD_{\min}:=\dom(A)$. Since both $A$ and $A^*$ are closed linear operators, the vector spaces $\cD_{\max}$ and $\cD_{\min}$ equipped with the $A^*$-graph inner products are complete. We set $\cD_{\max}^{\cG}:=\big(\cD_{\max}, (\cdot,\cdot)_{\cG} \big)$ and $\cD_{\min}^{\cG}:=\big(\cD_{\min}, (\cdot,\cdot)_{\cG} \big)$, where $(x,y)_{\cG}=(x,y)_{\cH}+(A^*x,A^*y)_{\cH}$ for $x,y\in \cD_{\max}$. The quotient space $\big(\cD^{\cG}_{\max}/\cD^{\cG}_{\min}, |||\cdot|||\big)$ equipped with norm $|||\gamma(x)|||:=\inf\big\{\|x+y\|_{\cH}\big|\,y\in \cD_{\min} \big\}$ is  a Banach space; here and below $\gamma$ is the natural projection, that is,
\begin{equation}
\gamma:\cD^{\cG}_{\max}\rightarrow \cD^{\cG}_{\max}/\cD^{\cG}_{\min},\ \gamma(x)=[x],\end{equation} 
where
$[x]\in \cD^{\cG}_{\max}/\cD^{\cG}_{\min}$ is the class corresponding to the vector $x\in \cD^{\cG}_{\max}$. Next, we consider the Hilbert space \[B:=(\cD_{\min})^{\perp, (\cdot, \cdot)_\cG}=\{y\in \cD^{\cG}_{\max}\big| (x,y)_{\cG}=0 \text{\ for all\ } x\in \cD^{\cG}_{\min} \}\]
 equipped with the $A^*$-graph scalar product $(\cdot, \cdot)_{\cG}$ and introduce the isometry $i\in \cB(\cD^{\cG}_{\max}/\cD^{\cG}_{\min}, B)$ between these two spaces by letting $i:\gamma(x)\mapsto P_{B}x$, where $P_{B}\in \cB(\cD_{\max}^{\cG})$ is the orthogonal projector onto $B$. Obviously, $i \gamma(x)$ does not depend on a particular representative of the class $\gamma(x)=[x]$, moreover, $|||\gamma(x)|||=\|P_Bx\|_{\cG}.$ The map $i$ induces an inner product on $\cD_{\max}/\cD_{\min}$ that agrees with the norm  $|||\cdot|||$.
We introduce the real Hilbert space $\cH_A$ by letting 
\begin{equation}
\cH_{A}:=(\cD^{\cG}_{\max}/\cD^{\cG}_{\min},\langle\cdot, \cdot\rangle),
\text{ where }
\langle\gamma(x),\gamma(y)\rangle:=(i\gamma(x),i\gamma(y))_{\cG}.
\end{equation} The Hilbert space $\cH_A$ 
is called the space of abstract boundary values. It has a natural
 symplectic form
\begin{align}
&\omega:\cH_A\times \cH_A \rightarrow \R,\,
\omega(\gamma(x),\gamma(y)):=(A^*x,y)_{\cH}-(x,A^*y)_{\cH}.\lb{2.4form}
\end{align}  The space $\cH_A$ is described in \cite{BbF95}, where it is shown that the form $\omega$ is bounded and non-degenerate. As pointed out in \cite{BbF95}, the space of abstract boundary values of a given symmetric operator contains information about its self-adjoint extensions. In particular, it is shown that a symmetric extension $A_\cD$ of the operator $A$ with the domain $\cD$ is self-adjoint if and only if the abstract traces $\gamma(x)$ of vectors $x$ from the domain $\cD$ of the extension form a Lagrangian subspace. 
 
 \begin{theorem}[\cite{BbF95}] \lb{absfur}
 (i) A symmetric extension $A_\cD$ of $A$ is self-adjoint if and only if $\gamma(\cD)$ is a Lagrangian subspace of the space $\cH_A$ of abstract boundary values.
 
 (ii) Let $A_\cD$ be a self-adjoint and Fredholm extension of $A$. Then $\gamma\big(\ker(A^*)\big)$ is a Lagrangian subspace of the space $\cH_A$ of abstract boundary values, moreover, $\gamma\big(\ker(A^*)\big)\in F\Lambda(\gamma(\cD))$.
 
 (iii) Let $V\in C^1([0,1], \cB(\cH))$ be a differentiable path of bounded, self-adjoint operators. If the self-adjoint extension $A_\cD$ of the symmetric operator $A$ has compact resolvent and $\ker\big(A^*+V_t \big)\cap\cD_{\min}=\{0\}$ for all $t\in[0,1]$, then the mapping $\Upsilon_{A}:t\mapsto\gamma\big(\ker(A^*+V_t)\big)$ belongs to $C^1([0,1], F\Lambda(\gamma(\cD)))$. Also, if $V$ is continuous, then so is $\Upsilon_A$.
 \end{theorem}
 \begin{remark}\lb{rem2.4}
 Assuming hypotheses of Theorem \ref{absfur} {\it (iii)}, there exists a family of orthogonal projections $P:t\mapsto P_t\in\cB(\cD_{\max}^\cG)$, such that $\ran(P_t)=\ker(A^*+V_t)$ for all $t\in[0,1]$ and $P\in C^1\big([0,1], \cB(\cD_{\max}^\cG)\big)$. Moreover, according to \cite[Remark 3.5]{CJLS}, \cite[Section IV.1]{DK74}, \cite[Remark 6.11]{F04} for a given $t_0\in[0,1]$ there exists a smooth family of boundedly invertible operators $U: t\mapsto U_t\in\cB(\cD_{\max}^\cG)$ such that $U_tP_{t_0}=P_{t}U_t$ for all $t$ near $t_0$.
 \end{remark}
We are now ready to present our first principal result.
\begin{theorem}\lb{mormas}
(i) Let $A$ be a symmetric operator in a real Hilbert space $\cH$ and $A_\cD$ be its self-adjoint extension with the domain $\cD\subset\cH$. Let $V\in C^1([0,1], \cB(\cH))$ be a differentiable path of bounded, self-adjoint operators in $\cH$. Assume that $A_\cD$ has compact resolvent, and there exists $\lambda_{\infty}<0$, such that $\ker\big(A_\cD+V_t-\lambda\big)=\{0\}$ for all $t\in[0,1]$ and $\lambda<\lambda_{\infty}$, and $\ker\big(A^*+V_t-\lambda \big)\cap\cD_{min}=\{0\}$ for all $t\in[0,1], \lambda\in[\lambda_{\infty},0]$. Then for each $\tau\in(0,1)$
\begin{equation}\lb{mff1}
\mo(A_\cD+V_{\tau})-\mo(A_\cD+V_{1})=\mi(\gamma(\ker(A^*)+V_t)\big|_{\tau\leq t\leq 1},\gamma(\cD)).
\end{equation}
(ii) If, in addition to the assumptions in (i), for some $\tau\in(0,1)$ and all $t_*\in[\tau,1]$ such that $\ker\big(A_\cD+V_{t_*}\big)\not= \{0\}$ the restriction of the operator  $\dot{V_t}(t_*):=\frac {dV_t}{dt}{\big|_{t=t_*}}$ to the finite dimensional space $\ker\big(A_\cD+V_{t_*}\big)$  is positive definite then 
\begin{equation}\lb{mff2}
\mo(A_\cD+V_{\tau})-\mo(A_\cD+V_{1})=\sum\limits_{\tau< t \leq 1}\dim\big(\ker(A_\cD+V_t)\big).
\end{equation}
(iii) If, in addition to the assumptions in (i), for some $\tau\in(0,1)$ and all $t_*\in[\tau,1]$ such that $\ker\big(A_\cD+V_{t_*}\big)\not= \{0\}$  the restriction of the operator  $\dot{V_t}(t_*):=\frac {dV_t}{dt}{\big|_{t=t_*}}$ to the finite dimensional space $\ker\big(A_\cD+V_{t_*}\big)$ is negative definite then 
\begin{equation}\lb{mff3}
\mo(A_\cD+V_{\tau})-\mo(A_\cD+V_{1})=-\sum\limits_{\tau\leq t < 1}\dim\big(\ker(A_\cD+V_t)\big).
\end{equation}
\end{theorem}
 We will split the proof in several steps.
 We notice that $V_t(A_\cD-\lambda)^{-1}\in \cB_{\infty}(\cH)$ for $\lambda\in\C\setminus \spec(A_\cD)$, that is, $V_t$ is a relatively compact perturbation of $A_\cD$, and thus $A_\cD+V_t$ has purely discrete spectrum. For a fixed $\tau\in(0,1]$ we introduce the following one parameter family of bounded self-adjoint operators $W_s$  in $\cH$ and fix a parametrization $\Sigma=\cup_{j=1}^4\Sigma_j\to\Gamma=\cup_{j=1}^4\Gamma_j$, $s\mapsto (t(s),\lambda(s))$ of the square in Figure 1:
 \begin{align}
 & W_s:=V_{t(s)}-\lambda(s)I_{\cH},\, s\in \Sigma=\cup_{j=1}^4\Sigma_j, \lb{par1}\end{align}
  where
 \begin{align}
 &\lambda(s)=s,\, t(s)=\tau ,\, s\in\Sigma_1:=[\lambda_{\infty},0],\lb{par2}\\
 &\lambda(s)=0,\, t(s)=s+\tau ,\, s\in\Sigma_2:=[0,1-\tau ],\lb{par3}\\
 &\lambda(s)= -s+1-\tau ,\, t(s)= 1,\, s\in\Sigma_3:=[1-\tau ,1-\tau -\lambda_{\infty}],\lb{par4}\\
 &\lambda(s)=\lambda_{\infty},\, t(s)=-s+2-\tau -\lambda_{\infty},\lb{par5}\\
 &\hskip3cm s\in\Sigma_4:=[1-\tau -\lambda_{\infty}, 2(1-\tau )-\lambda_{\infty}].\no
 \end{align}
As in Theorem \ref{absfur} {\it (iii)}, we consider the continuous path
$\Upsilon_{A}$ of Lagrangian subspaces in the symplectic Hilbert space $\cH_A$ of abstract boundary values, corresponding to symmetric operator $A$, that is, the path $\Upsilon_{A}:s\mapsto\gamma\big(\ker(A^*+W_s)\big)$, and observe that $\Upsilon_{A}\big|_{{\Sigma_i}}\in C^1(\Sigma_i, F\Lambda(\gamma(\cD))),\ 1\leq i\leq 4.$ 
\begin{lemma}\lb{g1} Under the assumptions in Theorem \ref{mormas} (i) one has 
\begin{equation}\lb{g1f}
\mi(\Upsilon_{A}\big|_{{\Sigma_1}},\gamma(\cD))=-\mo(A_\cD+V_{\tau}).
\end{equation}
\end{lemma}
\begin{proof}
The proof of this lemma relies on the computations made in the proof of Theorem 5.1 in \cite{BbF95}, specifically, see formula (5.3) in \cite{BbF95}.
  
 Let $s_*\in (\lambda_{\infty}, 0)$ be a conjugate point, i.e. $\cK_{s_*}:=\Upsilon_A(s_*)\cap \gamma(\cD)\not=\{0\}$. By Lemma \ref{crossex} there exists a small neighbourhood $\Sigma_{s_*}\subset(\lambda_{\infty}, 0)$ of $s_*$ and a family $(s+s_*)\mapsto B_{(s+s_*)}$ in $C^1\big(\Sigma_{s_*}, \cB(\Upsilon_A(s_*), \Upsilon_A(s_*)^{\perp})\big),\ B_{s_*}=0_{\cH}$, such that $\Upsilon_A(s)=\{[u]+B_{s+s_*}[u]\big| [u]\in \Upsilon_A(s_*) \}$ for all $(s+s_*)\in \Sigma_{s_*}$. Let us fix $u_0\in \ker\big(A^*+W_{s_{*}}\big)$ with $W_s$ from \eqref{par1} and consider the family $[u_0]_s:=[u_0]+B_{(s+s_*)}[u_0]$ with small $|s|$. Since $$\ker\big(A^*+W_{(s+s_*)}\big)\cap \cD_{\min}=\ker\big(A^*+V_{\tau}-(s+s_*)I_{\cH}\big)\cap \cD_{\min}=\{0\},$$ for small $|s|$, there exists a unique $u_s\in\ker\big(A^*+W_{(s+s_*)}\big)$ such that $[u_s]=[u_0]_s$. Next, using \eqref{2.4form} and \eqref{par2} we calculate:
\begin{align}
\omega([u_0], B_{(s+s_*)}[u_0])=&\big(A^*u_0, u_{s}-u_0\big)_{\cH}-\big(u_0, A^*(u_{s}-u_0)\big)_{\cH}\no\\
=&\big((A^*+W_{s_*})u_0, u_{s}-u_0\big)_{\cH}-\big(u_0, (A^*+W_{s_*})(u_{s}-u_0)\big)_{\cH}\no\\
=&-\big(u_0, (A^*+W_{s_*})u_{s}\big)_{\cH}=-\big(u_0, (A^*+W_{(s+s_*)})u_{s}\big)_{\cH}\no\\
&-\big(u_0,(W_{s_*}-W_{(s+s_*)})u_{s}\big)_{\cH}=-\big(u_0,s u_{s}\big)_{\cH}.\
\end{align} 
 Assuming, for a moment, that the mapping $s\mapsto u_s$ is continuous, we proceed by evaluating  the crossing form
\begin{align}
 \cQ_{s_*,\gamma(\cD)}([u_0],[u_0])&:=\frac{d}{ds}\omega([u_0], B_{(s+s_*)}[u_0])\big|_{s=0}=\lim\limits_{s\rightarrow 0} \frac{\omega([u_0], B_{(s+s_*)}[u_0])}{s}\no\\
&=\lim\limits_{s\rightarrow 0} \frac{-\big(u_0,s u_{s}\big)_{\cH}}{s}=-\|u_0\|^2_{\cH}.\no
\end{align}
By Theorem \ref{masform} {\it (i)} we therefore have 
\begin{align}\lb{masgamma1}
\mi\big(\Upsilon_{A}\big|_{\Sigma_{s_*}},\gamma(\cD)\big)&=\sign\ \cQ_{s_*,\gamma(\cD)}=-\dim\big(\Upsilon_A(s_*)\cap \gamma(\cD)\big)\no\\
&=-\dim \ker\big(A_\cD+V_{\tau}-s_*\big).
\end{align}
Formula \eqref{masgamma1} holds for all crossings $s_*\in \Sigma_1$, thus, using \eqref{par2},
\begin{align}
\mi\big(\Upsilon_{A}\big|_{\Sigma_1}\big)&=\sum\limits_{\substack{\lambda_{\infty}<s<0:\\
																		\Upsilon_A(s)\cap \gamma(\cD)\not=\{0\}	}}\sign\ \cQ_{s,\gamma(\cD)}+n_+(\cQ_{0,\gamma(\cD)})\no\\
&=-\sum\limits_{\lambda_{\infty}<s<0}\dim \ker\big(A_\cD+V_{\tau}-s\big)=-\mo(A_\cD+V_{\tau}),\lb{2.12}
\end{align}
since $\cQ_{0,\gamma(\cD)}$ is negative definite and thus  $n_+(\cQ_{0,\gamma(\cD)})=0$, while \eqref{2.12} follows from the absence of eigenvalues less than $\lambda_{\infty}$, which is guaranteed by the assumption $\ker\big(A_\cD+V_t-\lambda\big)=\{0\}$ for all $t\in[0,1]$ and $\lambda<\lambda_{\infty}$.

To complete the proof we need to show the continuity of $s\mapsto u_s$ at $s=0$. In fact, we will show that this map is continuously differentiable. Let us denote $\cN_s:=\ker\big(A^*+W_{(s_{*}+s)}\big)$. Since for all $(s+s_*)\in\Sigma$ one has $\cN_s\cap\cD_{\min}=\{0\}$, for each class $[x]\in \gamma\big(\cN_s\big)$ there exists a unique representative $x_{\ker}\in\cN_s$, so that $[x_{\ker}]=[x]$ and the mapping $T_s:\gamma\big(\cN_s\big)\rightarrow \cN_s,\ [x]\mapsto x_{\ker}$ is well defined and $T_s\gamma\big|_{\cN_s}=Id_{\cN_s}, \gamma\big|_{\cN_s}T_s=Id_{\gamma(\cN_s)}$. Then one has $u_s=T_s\big(I_{\cH_A}+B_{(s+s_*)}\big)[u_0]$ with $B_{(s+s_*)}$ defined in the beginning of the proof of the lemma. The differentiability of the function $s\mapsto u_s$ will be shown once we can proof the differentiability of the family of bounded operators $s\mapsto \big(I_{\cH_A}+B_{(s+s_*)}\big)^{-1}T_s^{-1}$. Here $|s|$ is sufficiently small, so that $\big(I_{\cH_A}+B_{(s+s_*)}\big)^{-1}$ is well defined and bounded, while the boundedness of $T_s^{-1}$ follows using the opening mapping theorem. Since the function  $s\mapsto \big(I_{\cH_A}+B_{(s+s_*)}\big)$ is differentiable, the differentiability of the function 
$s\mapsto \big(I_{\cH_A}+B_{(s+s_*)}\big)^{-1}$ follows and it remains to prove differentiability of the function 
 $s\mapsto T_s^{-1}$. Recall that $T_s^{-1}=\gamma\big|_{\cN_s}$. Employing Remark \ref{rem2.4}, one obtains $T_s^{-1}=\gamma\big|_{\cN_s}=\gamma U_sP_{\cN_{s*}}U_s^{-1}$, where $P_{\cN_{s}}$ is the orthogonal projection onto $\cN_s$ in the space $\cD_{\max}^{\cG}$. Thus the function $s\mapsto T_s^{-1}$ is differentiable, implying the desired property of the map $s\mapsto u_s$.
\end{proof}
Similarly, one can show the monotonicity of the Maslov index on $\Sigma_3$ and prove the following result.
\begin{lemma}\lb{g3}
Under the assumptions in Theorem \ref{mormas} (i) one has 
\begin{equation}\lb{g3f}
\mi(\Upsilon_A\big|_{\Sigma_3},\gamma(\cD))=\mo(A_\cD+V_1).
\end{equation}
\end{lemma}
\begin{proof}
	As in Lemma \ref{g1}.
\end{proof}
\begin{lemma}\lb{g4}
Under the assumptions in Theorem \ref{mormas} (i) one has 
\begin{equation}\lb{g4f}
\mi(\Upsilon_A\big|_{\Sigma_4},\gamma(\cD))=0.
\end{equation}
\end{lemma}
\begin{proof}
If $s_*\in\Sigma_4$ is a crossing, then $\Upsilon_A(s_*)\cap \gamma(\cD)\not=\{0\}$, implying by \eqref{par5} that $\ker\big(A_\cD+V_{-s+2-\tau-\lambda_{\infty}}-\lambda_{\infty}\big)\not=\{0\},\ s\in \Sigma_4$. By the assumption the kernel is trivial, therefore there are no crossings on $\Sigma_4$ and $\mi(\Upsilon_A|_{\Sigma_4},\gamma(\cD))=0$.
\end{proof}

\begin{lemma}\lb{g2}
Under the assumptions in Theorem \ref{mormas} (i),(ii) one has 
\begin{equation}\lb{masg4}
\mi(\Upsilon_A|_{\Sigma_2},\gamma(\cD))=\sum\limits_{\tau< t \leq 1}\dim\big(\ker(A_\cD+V_t)\big).
\end{equation}
\end{lemma}

\begin{proof}
Let $s_*\in [0,1-\tau ]$ be a conjugate point, i.e. $\cK_{s_*}:=\Upsilon_A(s_*)\cap \gamma(\cD)\not=\{0\}$. Then by Lemma \ref{crossex} there exists a small neighbourhood $\Sigma_{s_*}\subset[0,1-\tau ]$ of $s_*$ and a family $(s+s_*)\mapsto B_{(s+s_*)}$ in $C^1\big(\Sigma_{s_*}, \cB(\Upsilon_A(s_*), \Upsilon_A(s_*)^{\perp})\big),\ B_{s_*}=0_{\cH}$, such that $\Upsilon_A(s)=\{[u]+B_{(s+s_*)}[u]\big| [u]\in \Upsilon_A(s_*) \}$ for all $(s+s_*)\in \Sigma_{s_*}$. Let us fix $u_0\in \ker\big(A^*+W_{s_{*}}\big)$ and consider the family $[u_0]_s:=[u_0]+B_{(s+s_*)}[u_0]$ with small $|s|$. Since $\ker\big(A^*+W_{(s+s_*)}\big)\cap \cD_{\min}=\ker\big(A^*+V_{\tau+s_*+s}\big)\cap \cD_{\min}=\{0\}$, for small $|s|$ there exists a unique $u_s\in\ker\big(A^*+W_{(s+s_*)}\big)$ such that $[u_s]=[u_0]_s$. Next,
\begin{align}
\omega([u_0], B_{(s+s_*)}[u_0])&=\big(A^*u_0, u_{s}-u_0\big)_{\cH}-\big(u_0, A^*(u_{s}-u_0)\big)_{\cH}\no\\
=&\big((A^*+W_{s_*})u_0, u_{s}-u_0\big)_{\cH}-\big(u_0, (A^*+W_{s_*})(u_{s}-u_0)\big)_{\cH}\no\\
=&-\big(u_0, (A^*+W_{s_*})u_{s}\big)_{\cH}=-\big(u_0, (A^*+W_{(s+s_*)})u_{s}\big)_{\cH}\no\\
&-\big(u_0,(W_{s_*}-W_{(s+s_*)})u_{s}\big)_{\cH}=\big(u_0,(V_{\tau+s_*+s}-V_{\tau+s_*}) u_{s}\big)_{\cH}.\no
\end{align} 
Since the mapping $s\mapsto u_s$ is continuous (this can be shown by the same arguments as in the proof of Lemma \ref{g1}), we proceed by evaluating  the crossing form:
\begin{align}
 \cQ_{s_*,\gamma(\cD)}([u_0],[u_0])&:=\frac{d}{ds}\omega([u_0], B_{(s+s_*)}[u_0])\big|_{s=0}=\lim\limits_{s\rightarrow 0} \frac{\omega([u_0], B_{(s+s_*)}[u_0])}{s}\no\\
&=\lim\limits_{s\rightarrow 0} \Big(u_0,,\frac{(V_{\tau+s_*+s}-V_{\tau+s_*}) u_{s}}{s}\Big)_{\cH}=\big(u_0, \dot{V}_{\tau+s_*}u_0\big)_{\cH}>0.\no
\end{align}
By Theorem \ref{masform} {\it (i)} we then have 
\begin{align}\lb{masgamma11}
\mi\big(\Upsilon_{A}\big|_{\Sigma_{s_*}},\gamma(\cD)\big)&=\sign\ \cQ_{s_*,\gamma(\cD)}=\dim\big(\Upsilon_A(s_*)\cap \gamma(\cD)\big)\no\\
&=\dim \ker\big(A_\cD+V_{\tau+s_*}\big).
\end{align}
Formula \eqref{masgamma11} holds for all crossings $s_*\in \Sigma_2$, thus
\begin{align}
\mi\big(\Upsilon_{A}\big|_{\Sigma_2}\big)&=\sum\limits_{\substack{0<s<1-\tau:\\
																		\Upsilon_A(s)\cap \gamma(\cD)\not=\{0\}	}}\sign\ \cQ_{s,\gamma(\cD)}+n_+(\cQ_{1-\tau,\gamma(\cD)})-n_-(\cQ_{0,\gamma(\cD)})\no\\
&=\sum\limits_{0<s\leq 1-\tau}\dim \ker\big(A_\cD+V_{\tau+s}\big)\lb{2.17},
\end{align}
where $n_-(\cQ_{0,\gamma(\cD)})=0$, since $\cQ_{0,\gamma(\cD)}$ is positive definite, and $n_+(\cQ_{1-\tau,\gamma(\cD)})=\dim \ker\big(A_\cD+V_{1}\big)$, since $\cQ_{1,\gamma(\cD)}$ is positive definite.
\end{proof}
A similar argument gives the following result.
\begin{lemma}\lb{g2n}
	Under the assumptions in Theorem \ref{mormas} (i),(iii) one has 
	\begin{equation}\lb{masg4n}
	\mi(\Upsilon_A\big|_{\Sigma_2},\gamma(\cD))=-\sum\limits_{\tau\leq t <1}\dim\big(\ker(A_\cD+V_t)\big).
	\end{equation}
\end{lemma}
\begin{proof}[Proof of Theorem \ref{mormas}]
We recall that the Maslov index is additive under catenation \cite[Theorem 3.6(a)]{BbF95}, so that 
\begin{align*}
\mi(\Upsilon_{A},\gamma(\cD))&= \mi(\Upsilon_{A}\big|_{\Sigma_1},\gamma(\cD))+\mi(\Upsilon_{A}\big|_{\Sigma_2},\gamma(\cD))\\&+\mi(\Upsilon_{A}\big|_{\Sigma_3},\gamma(\cD))+\mi(\Upsilon_{A}\big|_{\Sigma_4},\gamma(\cD)).\end{align*} Moreover, $\mi(\Upsilon_{A}\gamma(\cD))=0$ since the Maslov index is homotopy invariant \cite[Theorem 3.6(b)]{BbF95}, so that we arrive at the identity
\begin{equation}\lb{mas0}\begin{split}
\mi(\Upsilon_{A}\big|_{\Sigma_1},\gamma(\cD))&+\mi(\Upsilon_{A}\big|_{\Sigma_2},\gamma(\cD))\\&+\mi(\Upsilon_{A}\big|_{\Sigma_3},\gamma(\cD))+\mi(\Upsilon_{A}\big|_{\Sigma_4},\gamma(\cD))=0.\end{split}
\end{equation}
We are ready to prove assertions {\it (i) -- (iii)} in Theorem \ref{mormas}.

{\it (i)}\, Equalities \eqref{mas0}, \eqref{g1f}, \eqref{g3f} and \eqref{g4f} yield
\begin{align*}
\mi(\Upsilon_{A}\big|_{\Sigma_2},\gamma(\cD))=&-\mi(\Upsilon_{A}\big|_{\Sigma_1},\gamma(\cD))-\mi(\Upsilon_{A}\big|_{\Sigma_3},\gamma(\cD))\\&-\mi(\Upsilon_{A}\big|_{\Sigma_4},\gamma(\cD))=\mo(A_\cD+V_{\tau})-\mo(A_\cD+V_{1}),
\end{align*} and imply \eqref{mff1}. 

{\it (ii)}\, Equalities \eqref{mas0}, \eqref{g1f}, \eqref{g3f}, \eqref{g4f} and\eqref{masg4}  yield \eqref{mff2}.

{\it (iii)}\, Equalities  \eqref{mas0}, \eqref{g1f}, \eqref{g3f}, \eqref{g4f} and\eqref{masg4n}  yield \eqref{mff3}. 
\end{proof}

\begin{corollary}
	
(i) If the assumptions in Theorem \ref{mormas} (i) and (ii) hold, then the function $\tau\mapsto \mo\big(A_\cD+V_{\tau}\big)$ is non-increasing
for $\tau\in(0,1]$. 

(ii) If the assumptions in Theorem \ref{mormas} (i) and (iii) hold, then the function $\tau\mapsto \mo\big(A_\cD+V_{\tau}\big)$ is non-decreasing for $\tau\in(0,1]$. 
\end{corollary}
\begin{proof}
Since the right hand sides of \eqref{mff2} and \eqref{mff3} are sign definite, assertions {\it (i)}  and {\it (ii)} follow from \eqref{mff2} and \eqref{mff3}.
\end{proof}
\section{Applications}\lb{applic}
\subsection{Schr\"{o}dinger operator with $\vec{\theta}$ periodic boundary conditions}

Given $n$ linearly independent vectors $\{a_1, \dots a_n\}\subset\R^n$, we consider the lattice $L$ and the dual lattice $L'$,
\begin{align}
L&:=\{k_1{a}_1+\cdots+k_n{a}_n|\  k_j\in\bbZ , j\in\{1,\dots, n\}\},\lb{dfnL}\\
L'&:=\{k_1{b}_1+\cdots+k_n{b}_n|\ {b}_j\cdot {a}_i=2\pi\delta_{ij},\ k_j\in\bbZ,\  i,j\in\{1,\dots, n\}\},\no
\end{align}
and the unit cell $Q$ and the dual unit cell $Q'$,
\begin{align*}
Q&:=\{t_1{a}_1+\cdots+t_n{a}_n|\ 0\leq t_j\leq 1 , j\in\{1,\dots, n\}\},\\
Q'&:=\{t_1{b}_1+\cdots+t_n{b}_n| {b}_j\cdot {a}_i=2\pi\delta_{ij},\ \ 0\leq t_j\leq 1 , i,j\in\{1,\dots, n\}\}.\end{align*}
 The faces $\partial Q^s_j$ of the unit cell $Q$
 (so that $\partial Q=\cup_{s=0}^1\cup_{j=1}^n\partial Q^s_j$) are denoted by \[{\partial Q}^{s}_{j}: =\{t_1{a}_1+\cdots+t_n{a}_n\in Q\big|\, t_{j}=s\}, \  j\in\{1,\dots, n\},\ s\in\{0,1\}.\] The $n$-tuple $\{{a}_1, \dots {a}_n\}\subset\R^n$ is uniquely associated with an $n\times n$ matrix $A$ by the condition $A{a}_j=2\pi {e}_j$, where $\{{e}_j\}_{1\leq j\leq n}$ is the standard basis in $\R^n$. Furthermore, the equalities $A^{\top}{e}_i\cdot {a}_j={e}_i\cdot A {a}_j=2\pi {e}_i\cdot {e}_j=2\pi\delta_{ij},\ i,j\in\{
1,\dots, n\}$ imply ${b}_j=A^{\top}{e}_j$. 
 For a given vector $\vec{\theta}:=(\theta_1,\dots, \theta_n)\in[0,1)^n$, the matrix $A$ just defined, and $k\in\bbZ^n$ we denote 
 \begin{align*}
 c_k(x)&:= |Q|^{-1}\cos\big(A^{\top}(\vec{\theta}-k)\cdot x\big),\,
 s_k(x):=|Q|^{-1} \sin\big(A^{\top}(\vec{\theta}-k)\cdot x\big),\\\zeta_k(x)&:=\e^{{\bfi A^{\top}(\vec{\theta}-k)\cdot x}},\,   x\in Q.
 \end{align*} 

We now briefly recall basic properties of the Dirichlet and Neumann trace operators, see \cite{GM10} and references therein for more details. The Dirichlet boundary trace operator $\gamma_D^0: C(Q)\rightarrow C(\partial Q),\ \gamma_D^0u:=u|_{\partial Q}$ is extended by continuity to a bounded operator $\gamma_D\in \cB\big(\Htworm,L^2(\partial Q;\R^{2m})\big)$ and called the Dirichlet trace operator. The traces of functions from $\Htworm$ have higher regularity but this is not significant for our purposes. Recalling that $\partial Q=\cup_{s=0}^{1}\cup_{j=1}^{n}{\partial Q}_j^s$, we define the Dirichlet trace operators corresponding to each face of $Q$, that is, we let
\begin{align}
&\gamma_{D,{\partial Q}_{j}^{s}}:\Htworm\rightarrow L_{2}({\partial Q}_{j}^{s},\R^{2m}),\no\\
&\gamma_{D,{\partial Q}_{j}^{s}}(u):=(\gamma_{D}u)|_{{\partial Q}_{j}^{s}},\ \ 1\leq j\leq n,\ s\in\{0,1\}.\no
\end{align}
It follows that $\gamma_{D,{\partial Q}_j^{s}}\in\cB\big(\Htworm,L^2({\partial Q}_j^{s};{\R}^{2m})\big)$ for $1\leq j\leq n$
and $s\in\{0,1\}$. We denote
\begin{align}
&\nabla u:=[\nabla u_1,\cdots,\nabla u_{2m}]^{\top}\in \R^{2m\times n},
\langle \nabla u(x),\nabla v(x)\rangle_{\R^{2m}}:=\sum\limits_{i=1}^{2m}(\nabla u_i(x),\nabla v_i(x))_{\R^n}\no
\end{align}
for given $u=(u_i)_{i=1}^{2m},v=(v_i)_{i=1}^{2m}\in \Htworm$. The Neumann trace is given by
\begin{align}
&\gamma_{N,{\partial Q}_{j}^{s}}:\Htworm\rightarrow L^2({\partial Q}_{j}^{s};{\R}^{2m}),\no\\
&\gamma_{N,{\partial Q}_{j}^{s}}(u):=\big({{\gamma}}_{D}(\nabla u) \overrightarrow{\nu}\big)\big|_{{\partial Q}_{j}^{s}},\ \ 1\leq j\leq n,\ s\in\{0,1\},\no
\end{align}
where $\vec{\nu} \text{\ is the outward pointing normal unit vector to\ }\partial Q$. 
The inclusion   $\gamma_{N,{\partial Q}_j^{s}}\in \cB\big(\Htworm,L^2({\partial Q}_j^{s};\R{^{2m\times n}})\big)$ holds for all $1\leq j\leq n,\ s\in\{0,1\}$. For each $u\in \Htworm$ we denote
\begin{equation}\label{notu10}
u_{j}^{s}:=\gamma_{D,{\partial Q}_{j}^{s}}(u),\quad \partial_{j}u^{s}:=\gamma_{N,{\partial Q}_{j}^{s}}(u),  \ 1\leq j\leq n,\ s\in\{0,1\}.
\end{equation} 

Our next objective is to discuss the Laplace operator satisfying $\vec{\theta}$-periodic boundary conditions $u(x+a_j)=\e^{2\pi\bfi\theta_j}u(x)$, $\frac{\partial u}{\partial \vec{\nu}}(x+a_j)=\e^{2\pi\bfi\theta_j}\frac{\partial u}{\partial \vec{\nu}}(x)$ for all $x\in {\partial Q}^0_{j}$,$\ 1\leq j\leq n.$ Since these conditions are complex, we will 
have to distinguish between the Laplace operator acting in the real and complex spaces (see Theorems \ref{operdef} and \ref{operdef1} respectively). We begin with the real case.
 
The Laplace operator $-\Delta: C_{0}^{\infty}(Q;\R)\rightarrow C_{0}^{\infty}(Q;\R)$ considered in $\Ltwor$ is closable. The closure of this operator is denoted by $-\Delta_{\min}$ and is given by
\begin{align}
&\dom(-\Delta_{\min})=\Htwor,\, -\Delta_{\min}u=-\Delta u,\end{align}
 in the sense on distributions;
 the adjoint operator  $-\Delta_{\max}$ is defined by
\begin{align}
&\dom(-\Delta_{\max}):=\{u\in \Ltwor|\  \Delta u\in \Ltwor\},\,
-\Delta_{\max}u:=-\Delta u,\no
\end{align}
so that one has $-\Delta_{\max}=(-\Delta_{\min})^*$ and $-\Delta_{\min}=(-\Delta_{\max})^*$.
The direct sum of $2m$ copies of the operator $-\Delta_{\min}$ is denoted again by  $-\Delta_{\min}:=\oplus_{j=1}^{2m}(-\Delta_{\min})$, similarly  $-\Delta_{\max}:=\oplus_{j=1}^{2m}(-\Delta_{\max})$. 

Given a vector $\vec{\theta}=(\theta_1,\dots, \theta_n)\in(0,1]^n$, we will now describe the Laplace operator with the real valued version of the $\vec{\theta}$-periodic boundary conditions. Let us introduce the following matrices $M_j$ and the weighted translation operators $\cM_j$:
\begin{align}\lb{dfnMM}
M_{{j}}&:=I_m \otimes \begin{bmatrix}
\cos2\pi\theta_{j}& -\sin2\pi\theta_{j} \\
\sin2\pi\theta_{j}&\cos2\pi\theta_{j}
\end{bmatrix},\, 1\leq j\leq n,\\
\cM_j&\in\cB (L^2({\partial Q}_j^0;\R^{2m}), L^2({\partial Q}_j^1;\R^{2m})),\no\\
(\cM_j u)(x)&=M_ju(x-a_j) \text{\ for a.a. } x\in{\partial Q}_j^1,\ 1\leq j\leq n.\lb{dfncM}
\end{align}
Observe that $\cM_j$ is an isometric isomorphism for each $1\leq j\leq n$:
\begin{align*}
(\cM_ju,\cM_jv)_{L^2({\partial Q}_j^1;\R^{2m})}=(u,v)_{L^2({\partial Q}_j^0;\R^{2m})},
\end{align*}
where we used $M_j^{\top}M_j=I_{2m}$.
The following result describes the Laplacian with the $\vec{\theta}$-periodic boundary conditions acting on the space of real vector valued functions. We recall notations \eqref{notu10} and \eqref{dfncM}. 
\begin{theorem}\lb{operdef}The $\vec{\theta}$-periodic real Laplace operator $-\Delta_{\vec{\theta},\R}$ defined by
\begin{align}
 -\Delta_{\vec{\theta},\R}&:\dom (-\Delta_{\vec{\theta}, \R})\subset \Ltworm\rightarrow \Ltworm,\no\\
 \dom (-\Delta_{\vec{\theta}, \R})&:=\Big\{u \in H^2(Q; \R^{2m}) \Big|\ u^1_j=\cM_ju^0_{j},\ \partial_ju^1=-\cM_j\partial_{j}u^0,\ 1\leq j\leq n\  \Big\},\no\\
-\Delta_{\vec{\theta}, \R} u&:=-\Delta u, \ u\in  \dom (-\Delta_{\vec{\theta}, \R}),\no
\end{align}
is self-adjoint, moreover,
$-\Delta_{\min}\subset-\Delta_{\vec{\theta}, \R}\subset-\Delta_{\max}$.
In addition, the operator $-\Delta_{\vec{\theta}, \R}$ has compact resolvent, in particular, it has purely discrete spectrum. Finally, $\spec(-\Delta_{{\vec{\theta}},\R})=\big\{\|A^{\top}(\vec{\theta}-k)\|^2_{\R^n}\big\}_{k\in\bbZ^n}$.
\end{theorem}
The proof of this theorem relies on several auxiliary results. We begin with a simple abstract lemma.
\begin{lemma}\lb{abs} Let $\cH$ be separable Hilbert space (complex or real) with the inner product $(\cdot, \cdot)_{\cH}$ and corresponding norm $\|\cdot\|_{\cH}$.
Consider a linear operator $S$ such that $(Sx,y)_{\cH}=(x,Sy)_{\cH}$  for all $x,y\in\dom{S}$. Assume that there exists a complete orthonormal sequence  $\{\phi_k\}_{k=1}^{\infty}$ of eigenvectors of $S$. Then $S$ is essentially self-adjoint, i.e. $\overline{S}=S^*$. Let $\phi_k$  be the eigenvector corresponding to the eigenvalue $\lambda_k$. Then
\begin{equation}
\lb{opdef}
\dom(\overline{S})=\big\{ f\in\cH\colon \sum\limits_{k=1}^{\infty}\lambda_k^2\big|(f,\phi_k)_{\cH}|^2<\infty \big\},
\end{equation} 
and for each $f\in\dom{(\overline{S})}$ we have
\begin{equation}
\lb{opdef1}
\overline{S}f=\sum\limits_{k=1}^{\infty}\lambda_k(f,\phi_k)_{\cH}\phi_k.
\end{equation}
Finally, if $\lim\limits_{k\rightarrow+\infty}|\lambda_k|=+\infty$, then $\overline{S}$ has compact resolvent and $\spec(\overline{S})=\spec_d(\overline{S})=\{\lambda_k\}_{k\in\bbN}$.
\end{lemma}
\begin{proof}
In order to show that $S$ is closable we pick a sequence of vectors $ \{ f_k\}_{k=1}^{\infty}\subset \dom{(S)}$ such that  $\lim\limits_{k\rightarrow \infty} {f_k}=0$ and 
$\lim\limits_{k\rightarrow \infty} {Sf_k}=g$, for some $g\in\cH$. For each $l\in \bbN$ the following chain of equalities holds
\begin{align}
\left(g,\phi_l\right)_{\cH}&=\big(\lim\limits_{k\rightarrow \infty} {Sf_k},\phi_l\big)_{\cH}=\lim\limits_{k\rightarrow \infty}\big(Sf_k,\phi_l\big)_{\cH}\no \\
&=\lim\limits_{k\rightarrow \infty} (f_k,S\phi_l)_{\cH}=\big(\lim\limits_{k\rightarrow \infty} f_k,S\phi_l\big)_{\cH}=0,\no
\end{align}
thus $g=0.$ Next we  prove \eqref{opdef}. Pick a vector $f\in\dom{(\bar{S})}$, by definition of $\bar{S}$ there exists $ \{ f_k\}_{k=1}^{\infty}\subset \dom{S}$ such that $\  \lim\limits_{k\rightarrow \infty} {f_k}=f \text{\ and}\lim\limits_{k\rightarrow \infty} {Sf_k}=\bar{S}f.$
By 
\begin{align}
(\bar{S}f,\phi_l)_{\cH}&=\big(\lim\limits_{k\rightarrow \infty} {Sf_k},\phi_l\big)_{\cH}=\lim\limits_{k\rightarrow \infty}(Sf_k,\phi_l)_{\cH}=\no \\
&=\lim\limits_{k\rightarrow \infty} (f_k,S\phi_l)_{\cH}=\lambda_l\big(\lim\limits_{k\rightarrow \infty} f_k,\phi_l\big)_{\cH}=\lambda_l(g,\phi_l)_{\cH}, \ l\in \bbN
\end{align}
and Parseval's equality one obtains
\begin{align}
\lb{par}
&\|\bar{S}f\|_{\cH}^2= \sum\limits_{n=1}^{\infty}|(\bar{S}f,\phi_n)_{\cH}|^2=\sum\limits_{n=1}^{\infty}\lambda_n^2|(f,\phi_n)_{\cH}|^2.
\end{align}
Since $\|\bar{S}f\|_{\cH}^2<\infty$, the right hand side of \eqref{par} is finite, therefore,"$\subseteq$" in \eqref{opdef} is shown. To show the opposite inclusion, let $f$ be such that $\sum\limits_{n=1}^{\infty}\lambda_n^2|(f,\phi_n)_{\cH}|^2<\infty$ and set $f_k:=\sum\limits_{r=1}^{k}(f,\phi_r)_{\cH}\phi_r$. Then $f_k\in\dom{S}$  and $\lim\limits_{k\rightarrow \infty} f_k=f$, also  $Sf_k=\sum\limits_{r=1}^{k}\lambda_n(f,\phi_r)_{\cH}\phi_r$  and the sequence $\{Sf_k\}_{k=1}^{\infty}$ converges. This proves the inclusion ``$\supseteq$" in \eqref{opdef} and also yields \eqref{opdef1}. Let $a\in\R\backslash\spec(\overline{S})$, then $$(\overline{S}-a)^{-1}f=\sum\limits_{k=1}^{\infty}(\lambda_k-a)^{-1}(f,\varphi_k)_{\cH}\varphi_k, \text{\ for \ } f\in\cH.$$
Finally, the inclusion $(\overline{S}-a)^{-1}\in\cB_{\infty}(\cH)$ follows since $\lim\limits_{k\rightarrow\infty}{(\lambda_k-a)^{-1}}=0$.
\end{proof}
 We prove that the eigenfunctions of the Laplacian form an orthonormal basis.
\begin{lemma}\lb{complexbasis}
	For each vector $\vec{\theta}:=(\theta_1,\dots, \theta_n)\in[0,1)^n$, the sequence of functions $\left\{{|Q|}^{-1}{\e^{\bfi  A^{\top}(\vec{\theta}-k)\cdot x}} \right\}_{k\in\bbZ^n}$ is orthonormal basis in $L^2(Q;\C)$.
\end{lemma}

\begin{proof}
	Suppose that $f\in L^2(Q;\C)$ is orthogonal to each function from the sequence. For all $k\in\bbZ^n$, change of variables $y=(2\pi)^{-1}Ax$ yields
	\begin{align}
	0&=\int\limits_{Q}\overline{f(x)}{\e^{\bfi  A^{\top}(\vec{\theta}-k)\cdot x}}d^nx=\int\limits_{Q}\overline{f(x)}{\e^{\bfi  (\vec{\theta}-k)\cdot Ax}}d^nx\no\\
	&=\frac{(2\pi)^n}{|\det A|}\int\limits_{[0,1]^n}\overline{f(2\pi A^{-1}y)}\e^{2\pi\bfi\vec{\theta}\cdot y}{\e^{-2\pi \bfi k\cdot y}}d^ny.\no
	\end{align}
	Thus, $\overline{f(2\pi A^{-1}y)}\e^{2\pi\bfi\vec{\theta}\cdot y}=0$ a.e. in $[0,1]^n$, implying $f(x)=0$ a.e. in $Q$.
\end{proof}

 A function $\phi:\R^n\rightarrow \C$ is called $\vec{\theta}-$periodic, if $\phi(x+{a}_j)=\e^{2\pi \bfi \theta_j}\phi(x)$ for all $x\in Q,\ 1\leq j\leq n$. The equalities $A^{\top}(\vec{\theta}-k)\cdot {a}_j=(\vec{\theta}-k)\cdot A {a}_j=2\pi(\vec{\theta}-k)\cdot {e}_j=2\pi(\theta_j-k_j),$ $k\in \bbZ^n,1\leq j\leq n$ yield
\begin{align}
& {|Q|}^{-1}{\e^{\bfi  A^{\top}(\vec{\theta}-k)\cdot (x+{a}_j)}}=\e^{2\pi \bfi \theta_j}{|Q|}^{-1}{\e^{\bfi  A^{\top}(\vec{\theta}-k)\cdot x}},\ 1\leq j\leq n,\no
\end{align}
moreover, for each vector  $v\in \R^n$, 
\begin{align}
& v\cdot \nabla\big({|Q|}^{-1}{\e^{\bfi  A^{\top}(\vec{\theta}-k)\cdot (x+{a}_j)}}\big)=\e^{2\pi \bfi \theta_j}v\cdot\nabla\big({|Q|}^{-1}{\e^{\bfi  A^{\top}(\vec{\theta}-k)\cdot x}}\big),\ 1\leq j\leq n.\no
\end{align}
Thus, we observe that the functions ${|Q|}^{-1}{\e^{\bfi  A^{\top}(\vec{\theta}-k)\cdot x}}$ and $ v\cdot \nabla\big( {|Q|}^{-1}{\e^{\bfi  A^{\top}(\vec{\theta}-k)\cdot x}}\big)$ are  $\vec{\theta}$-periodic.
Also, we recall the notation $c_k(x):=|Q|^{-1}\Re\e^{\bfi A^{\top}(\vec{\theta}-k)\cdot x}$ and $s_k(x):=|Q|^{-1}\Im\e^{\bfi A^{\top}(\vec{\theta}-k)\cdot x}$.
\begin{lemma}
\lb{realbasis}
The sequence $\big(\big(c_k(x),s_k(x)\big)^{\top}, \big(-s_k(x),c_k(x)\big)^{\top}\big)_{k\in\bbZ^n}$ is an orthonormal basis of $L^2(Q;\R^2)$.
\end{lemma}

\begin{proof}
	If $(f,g)^{\top}$ is orthogonal to each element of the sequence, then 
	\begin{align}
	0&=\int\limits_{Q}(c_k(x)f(x)+s_k(x)g(x))\ d^nx=\Re\int\limits_{Q}(c_k(x)+\bfi s_k(x)(f(x)-\bfi g(x))\ d^nx\no\\
	&=\Re\left(\e^{ \bfi A^{\top}(\vec{\theta}-k) \cdot x},f(x)+\bfi g(x)\right)_{\Ltwo},\no\\
	0&=\int\limits_{Q}(s_k(x)f(x)-c_k(x)g(x))\ d^nx=\Im\int\limits_{Q}(c_k(x)+\bfi s_k(x))(f(x)-\bfi g(x))\ d^nx\no\\&=\Im\left(\e^{\bfi A^{\top}(\vec{\theta}-k) \cdot x},f(x)+\bfi g(x)\right)_{\Ltwo}\no.
	\end{align}
Thus, we conclude that $(\e^{\bfi A^{\top}(\vec{\theta}-k) \cdot x},f(x)+\bfi g(x))_{\Ltwo}=0$ for all ${k\in\bbZ^n}$.
Since $\{\e^{\bfi A^{\top}(\vec{\theta}-k) \cdot x}\}_{k\in\bbZ^n}$ is an  orthonormal basis in $\Ltwo$ by Lemma \ref{complexbasis}, we conclude that $f(x)=g(x)=0$ a.e. in $Q$.
\end{proof}

For $k\in\bbZ^n$ and $1\leq l\leq m$ we denote 
\begin{align}
&\phi_{k,l}(x):=(0,\cdots, c_k(x), s_k(x),\cdots, 0)^{\top},\lb{bas1}\\
&\psi_{k,l}(x):=(0,\cdots, -s_k(x), c_k(x),\cdots, 0)^{\top}\lb{bas2},
\end{align}
where the nonzero terms are located at the positions $2l-1$ and $2l$.
\begin{corollary}\lb{basis}
 The sequence $\{\phi_{k,j},\psi_{k,j}\}_{k\in \bbZ^n,\ 1\leq j\leq n}$ is an orthonormal basis in $\Ltworm$.
\end{corollary}
\begin{proof} Apply Lemma \ref{realbasis}.\end{proof}
Next, we introduce the operator $T$ whose closure is the $\vec{\theta}$-periodic Laplacian.
\begin{lemma}\lb{optT}
The operator
 \begin{align}
 &T:\dom (T)\subset \Ltworm\rightarrow \Ltworm,\no\\
 &\dom (T):=\Big\{c_1\phi_{k,l}+c_2\psi_{k,l}+c_3\eta\Big|\no\\ 
 &\qquad\qquad\quad\ \ c_i\in \R,\ \eta \in C^\infty_{0}(Q;\R^{2m}),\ k\in\bbZ^n,\ 1\leq l \leq m,\ 1\leq i\leq 3\Big\},\no\\
 & Tu:=-\Delta_{\max} u, \ u\in  \dom (T)\no
 \end{align}
is densely defined, the equality $(Tu,v)_{\Ltwor^{2m}}=(u,Tv)_{\Ltwor^{2m}}$ holds for all $u,v \in \dom(T)$; moreover, for each $k\in\bbZ^n$ and $1\leq l\leq m$ the functions $\phi_{k,l}(x)$ and $\psi_{k,l}(x)$ are the eigenfunctions of $T$ corresponding to the eigenvalues $\|A^{\top}(k-\vec{\theta})\|^2_{\R^n}$.
\end{lemma}

\begin{proof} The operator $T$ is densely defined, since $C^\infty_{0}(Q;\R^{2m})\subset \dom(T)$. 

For $u,v \in \dom(T)$, integrating by parts twice
 \begin{align}
 &(Tu,v)_{(\Ltwor)^{2m}}=(-\Delta_{\max}u,v)_{(\Ltwor)^{2m}}=\int\limits_{Q}(-\Delta u, v)_{\R^{2m}}d^nx\no\\
 &=-\int\limits_{\partial Q}((\nabla u)\overrightarrow{\nu}, v)_{\R^{2m}}d^{n-1}x+\int\limits_{Q}\langle\nabla u,\nabla v\rangle_{\R^{2m}}d^nx=\int\limits_{Q}\langle\nabla u,\nabla v\rangle_{\R^{2m}}d^nx\no\\
 &=(u,Tv)_{(\Ltwor)^{2m}},
 \end{align}
 yields $(Tu,v)_{\Ltwor^{2m}}=(u,Tv)_{\Ltwor^{2m}}$. Indeed, in order to prove the equality $\int\limits_{\partial Q}(\nabla u)\overrightarrow{\nu}, v)_{\R^{2m}}d^{n-1}x=0$, we notice that $\cM_jw_j^0=w_j^1$ and $-\cM_j\partial_jw^0=\partial_jw^1$ for all $w\in \dom(T)$ and $1\leq j\leq n$. Since $\cM_j$ is an isometry for all $1\leq j\leq n$, we have
 \begin{align}
 &\int\limits_{\partial Q}( (\nabla u)\overrightarrow{\nu}, v)_{\R^{2m}}d^{n-1}x=\sum\limits_{j=1}^{n}\Bigg(\int\limits_{{\partial Q}_{j}^1}(\partial_{j} u^1(x), v_{j}^1(x))_{\R^{2m}}d^{n-1}x\no\\ &\hspace{21em}+ \int\limits_{{\partial Q}_{j}^0}(\partial_{j} u^0(x), v_{j}^0(x))_{\R^{2m}}d^{n-1}x\Bigg)\no\\
 &=\sum\limits_{j=1}^{n}\Bigg(\int\limits_{{\partial Q}_{j}^1}(-(\cM_j\partial_j u^0)(x), (\cM_jv_j^0)(x))_{\R^{2m}}d^{n-1}x\no\\ &\hspace{21em}+
 \int\limits_{{\partial Q}_j^0}(\partial_j u^0(x), v_j^0(x))_{\R^{2m}}d^{n-1}x\Bigg)\no\\
 &=\sum\limits_{j=1}^{n}\Bigg(\int\limits_{{\partial Q}_j^0}(-\partial_ju^0(x), v_j^0(x))_{\R^{2m}}d^{n-1}x+
 \int\limits_{{\partial Q}_j^0}(\partial_j u^0(x), v_j^0(x))_{\R^{2m}}d^{n-1}x\Bigg)=0.\no
 \end{align}
A straightforward differentiating yields
\begin{align}
&-\Delta_{\max}\phi_{k,l}(x)=\|A^{\top}(\vec{\theta}-k)\|^2_{\R^n}\phi_{k,l}(x),\,
-\Delta_{\max}\psi_{k,l}(x)=\|A^{\top}(\vec{\theta}-k)\|^2_{\R^n}\psi_{k,l}(x)\no
\end{align}
for all $x\in Q$, $k\in\bbZ^n$, $1\leq l\leq m$, and the last assertion in the lemma.
\end{proof}

\begin{proof}[Proof of Theorem \ref{operdef}] The operator $T$ from Lemma \ref{optT} satisfies assumptions of Lemma \ref{abs}, therefore it is essentially self-adjoint. We claim that  
$\dom(\overline{T})\subset \Htworm$. To prove this inclusion, we recall that $\dom(\overline{T})$ is equal to the closure of $\dom(T)$ with respect to the graph norm of $T$. By Poincare's inequality there exist positive constants $c_1(m,n,Q), c_2(m,n,Q)$ such that the inequalities 
\begin{align}
c_1(m,n,Q) \|\varphi\|_{(\Htworm)}^2&\leq\|\varphi\|_{(\Ltworm)}^2+\|-\Delta_{\max, 2m}\varphi\|_{(\Ltworm)}^2 \no\\
&\leq c_2(m,n,Q) \|\varphi\|_{(\Htworm)}^2\no
\end{align}
hold for all $\varphi\in\dom(T)$, thus the graph norm of the operator $T$ is equivalent to $\Htworm$-norm. Since $\dom(T)\subset \Htworm$, its closure with respect to $\Htworm$-norm is still a subset of $\Htworm$, i.e. $\dom(\overline{T})\subset \Htworm$. 
Since $-\Delta_{\min}\subset T\subset\overline{T}$ and $\overline{T}$ is self-adjoint, one has $-\Delta_{\min}\subset\overline{T}\subset(-\Delta_{\min})^*=-\Delta_{\max}.$
Next, we claim that $\overline{T}=-\Delta_{\theta, \R}$; assuming this claim we observe that Lemma \ref{abs} with $T=S$ and Lemma \ref{optT} imply all statements of Theorem \ref{operdef}. 

Starting the proof of the claim, let us prove $\overline{T}\subset-\Delta_{\theta, \R}$. Indeed, we observe that $\overline{T}\subset-\Delta_{\max}$ and $-\Delta_{\theta, \R}\subset-\Delta_{\max}$, therefore it is enough to prove the inclusion $\dom(\overline{T})\subset\dom(-\Delta_{\theta, \R})$. Since $\dom(\overline{T})\subset \Htworm$, it remains to show that functions from $\dom(\overline{T})$ satisfy the $\vec{\theta}$-periodic boundary conditions.  The latter follows from the continuity of the Dirichlet and Neumann trace operators. Indeed, recalling notation \eqref{dfncM},  for each $h\in \dom(T)$ one has $h^1_j=\cM_jh^0_j,\ \partial_jh^1=-\cM_j\partial_jh^0,\ 1\leq j\leq n$. For each $f\in\dom(\overline{T})$, there exists a sequence $h_r\in\dom(T)$, such that $f=\lim\limits_{r\rightarrow \infty}h_r$ in $\Htworm$. Finally, $\cM_j\in\cB (L^2({\partial Q}_j^0;\
R^{2m}), L^2({\partial Q}_j^1;\R^{2m}))$ 
implies $f^1_j=\cM_jf^0_j,\ \partial_jf^1=-\cM_j\partial_jf^0,\ 1\leq j\leq n $, and yields the required inclusion $\dom(\overline{T})\subset \dom(-\Delta_{\vec{\theta},\R})$. Since $-\Delta_{\theta, \R}$ is symmetric (this follows from integration by parts), $\overline{T}\subset-\Delta_{\theta, \R}$ and $\overline{T}$ is self-adjoint, we conclude that $\overline{T}=-\Delta_{\theta, \R}$.
\end{proof}
We will now discuss the Laplace operator acting on the space of complex valued functions. Analogously  to Theorem \ref{operdef}, one defines  the $\vec{\theta}$-periodic Laplacian acting on the complex Hilbert space $\Ltwom$. We adopt the same notations  $\gamma_D, \gamma_N$ for the Dirichlet and Neumann traces of complex valued functions and use notation \eqref{notu10}. Also, for the complex case, we introduce the translation operators $\bbM_j$ as follows, cf.\ \eqref{dfncM}:
\begin{align}
&\bbM_j\in\cB (L^2({\partial Q}_j^0;\C^m), L^2({\partial Q}_j^1;\C^m)),\no\\
&(\bbM_j u)(x)=\e^{2\pi\bfi\theta_j}u(x-a_j) \text{\ for a.a.\ }x\in{\partial Q}_j^1,\ 1\leq j\leq n.\label{dfnbM}
\end{align}

\begin{theorem}\lb{operdef1}The linear operator
	\begin{align}
	 -\Delta_{\vec{\theta}}&:\dom (-\Delta_{\vec{\theta}})\subset \Ltwom\rightarrow \Ltwom,\no\\
	 \dom (-\Delta_{\vec{\theta}})&:=\Big\{u \in\Htwom \Big|\,u^1_j=\bbM_ju^0_{j},\partial_ju^1=-\bbM_j\partial_{j}u^0,\ 1\leq j\leq n\  \Big\},\no\\
-\Delta_{\vec{\theta}} u&:=-\Delta u, \ u\in  \dom (-\Delta_{\vec{\theta}})\no
\end{align}
	is self-adjoint, moreover
$-\Delta_{\min}\subset-\Delta_{\vec{\theta}}\subset-\Delta_{\max}$.
	In addition, the operator $-\Delta_{\vec{\theta}}$ has compact resolvent, in particular, it has purely discrete spectrum. Finally, $\spec(-\Delta_{\vec{\theta}})=\big\{\|A^{\top}(\vec{\theta}-k)\|^2_{\R^n}\big\}_{k\in\bbZ^n}$.
\end{theorem}
\begin{proof}
 This follows from Lemmas \ref{abs} and \ref{complexbasis} and arguments similar to the proof of Theorem \ref{operdef}.  
\end{proof}
Next, we prove a simple proposition relating the spectra and the real and complex dimensions of the kernels of the Schr\"odinger operators acting in the spaces of real and complex valued functions.
For a given $V: Q\rightarrow \R^{m\times m} \text{\ and\ } t\in\R,$ we recall the notation $V_{t,\R}(x):=t^2V(tx)\otimes I_2$, $V_{\R}(x):=V(x)\otimes I_2$, and $V_t(x):=t^2V(tx)$. We also denote
$J:=\left[\begin{smallmatrix}
0 & -1 \\1 & 0\end{smallmatrix}\right]$, $\cJ=I_m\otimes J$.
\begin{proposition}\lb{3.9}
Let $V\in L^{\infty}(Q;\R^{m\times m})$ be a function whose values are symmetric matrices with real valued entries. Then $\spec\big(-\Delta_{\vec{\theta}}+V\big)=\spec\big(-\Delta_{\vec{\theta},\R}+V_{\R}\big)$. Moreover, for all real $\lambda$ we have
\begin{equation}\lb{dimeq}
\dim_{\R}\ker\big(-\Delta_{\vec{\theta},\R}+V_{\R}-\lambda\big)=2\dim_{\C}\ker\big(-\Delta_{\vec{\theta}}+V-\lambda\big). 
\end{equation}

\end{proposition}

\begin{proof}
	Since the functions $V\in L^{\infty}(Q;\R^{m\times m})$ and $V_{\R}\in L^{\infty}(Q;\R^{2m\times 2m})$ take values in the set of symmetric matrices, by Theorems \ref{operdef} and \ref{operdef1} the operators $-\Delta_{\vec{\theta}}+V$ and $-\Delta_{\vec{\theta},\R}+V_{\R}$ are self-adjoint with the domains $\dom(-\Delta_{\vec{\theta}})$ and $\dom(-\Delta_{\vec{\theta},\R})$ correspondingly as bounded perturbations of $-\Delta_{\vec{\theta}}$ and $-\Delta_{\vec{\theta},\R}$.

Starting the proof of \eqref{dimeq}, we introduce the maps $\cC$, $\cC^{-1}$ between the real  and  complex finite dimentional spaces $\ker\big(-\Delta_{\vec{\theta},\R}+V_{\R}-\lambda\big)$ and $\ker\big(-\Delta_{\vec{\theta}}+V-\lambda\big)$:
\begin{align}
&\cC:\ker\big(-\Delta_{\vec{\theta},\R}+V_{\R}-\lambda\big)\rightarrow \ker\big(-\Delta_{\vec{\theta}}+V-\lambda\big),\lb{C1}\\
&\cC: u=(u_1,u_2,\dots,u_{2m-1},u_{2m})\mapsto (u_1+\bfi u_2,\dots,u_{2m-1}+\bfi u_{2m}).\lb{C2}\\
&\cC^{-1}: v=(v_1,\dots,v_{m})\mapsto (\Re v_1,\Im v_2,\dots,\Re v_{m}, \Im v_{m}).\lb{Cinv2}
\end{align} 
By the definition of $\cC, \cC^{-1} $ we have: 
\begin{align}
\cC u&=0 \text{\ if and only if \ } u=0,\lb{czero}\\
 \bfi\cC u&=\cC \cJ u,\,
\cJ M_j = M_j \cJ,\text{\ for each \ } 1\leq j \leq n .\lb{cj1}
\end{align}
To prove the  inequality ``$\leq$" in \eqref{dimeq}, we notice that $u\in\ker\big(-\Delta_{\vec{\theta},\R}+V_{\R}-\lambda\big)$ if and only if $\cJ u\in\ker\big(-\Delta_{\vec{\theta},\R}+V_{\R}-\lambda\big)$
since $\cJ$ commutes with $-\Delta_{\vec{\theta},\R}$ and $V_{\R}$. In addition, $u,\cJ u$ are linearly independent over $\R$. Iterating, we construct the basis $u_1, u_2,\dots, u_{\varkappa}, \cJ u_1,\cJ  u_2,\dots,\cJ  u_{\varkappa}$, where $2\varkappa$ is the dimension of the real vector space $\ker\big(-\Delta_{\vec{\theta},\R}+V_{\R}-\lambda\big)$. We claim that $\cC u_1, \cC u_2,\dots, \cC u_{\varkappa}$ are linearly independent over $\C$. Indeed, if $\alpha_k, \beta_k\in\R$ and\begin{equation}\lb{lincontr}
0=\sum\limits_{k=1}^{\varkappa} (\alpha_k+\bfi\beta_k)\cC u_k=\cC\big(\sum\limits_{k=1}^{\varkappa} (\alpha_k+\cJ\beta_k) u_k\big)
\end{equation}
then $\alpha_k=\beta_k=0$ by \eqref{czero} and \eqref{cj1}, and we have the desired inequality.
To prove the inequality "$\geq$" in \eqref{dimeq}, we choose a basis $v_1, \dots, v_\kappa$,  where $\kappa$ is the dimension of the complex vector space $\ker\big(-\Delta_{\vec{\theta}}+V-\lambda\big)$. For the set of vectors $\cC^{-1}v_1, \dots, \cC^{-1}v_{\kappa},\cJ \cC^{-1}v_{1},\dots, \cJ\cC^{-1}v_{\kappa}$ and any real $\mu_1,\dots, \mu_{2\kappa}$ if  
\begin{equation}\lb{lincontr1}
0=\sum\limits_{j=1}^{\kappa}\big( \mu_j\cC^{-1} v_j+\mu_{\kappa+j}\cJ \cC^{-1} v_j\big)=\cC^{-1}\big(\sum\limits_{j=1}^{\kappa}\big( \mu_j+\bfi\mu_{\kappa+j}\big) v_j\big).
\end{equation}
then $\mu_j=\mu_{\kappa+j}=0$ which implies the inequality and concludes the proof.
\end{proof}
In the sequel we will impose the following assumptions on the potential $V$.
\begin{hypothesis} \lb{phyp}
	Let $V:Q\rightarrow\R^{m\times m}$ be such that $V(x)^{\top}=V(x)$ for almost all $x\in Q$. We will impose the following assumptions:
	\begin{enumerate}
		\item [{\it (i)}] $V\in L^{\infty}(Q;\R^{m\times m})$,
		\item [{\it (ii)}] $V\in C^1(Q;\R^{m\times m})$,
		\item [{\it (iii)}] $V\in C^1(Q;\R^{m\times m})$  and $\min\spec\big(2tV(tx)+t^2\nabla V(tx)x\big)>0$ for each $t~\in~(0,1]$ and almost all $x\in Q$,
		\item [{\it (iv)}] $V\in C^1(Q;\R^{m\times m})$  and $\min\spec\big(2tV(tx)+t^2\nabla V(tx)x\big)<0$ for each $t~\in~(0,1]$ and almost all $x\in Q$,
		\item [{\it (v)}] the matrix $V(0)$ is invertible.
	\end{enumerate}
\end{hypothesis}	
\noindent	In all cases we denote $\|V\|_{L^{\infty}(Q)}:=\sup\limits_{x\in Q}\|V(x)\|_{\R^{m\times m}}$. Aslo, we denote $$\lambda_{\infty}:=-\sup\limits_{t\in[0,1]}\|V_{t,\R}\|_{L^{\infty}(Q)}-1.$$
The following simple spectral results are used below.
\begin{lemma}\lb{noegv} Assume Hypothesis \ref{phyp} {\it (i)}. Then 
	\begin{equation}
	\lb{noegvf}
	\spec\big(-\Delta_{\vec{\theta},\R}+V_{t,\R}\big)\cap(-\infty, \lambda_\infty)=\emptyset,
	\end{equation}
	for all $t\in[0,1]$. In particular, \begin{equation}
\lb{cnoegvf}
\ker\big(-\Delta_{\vec{\theta},\R}+V_{t,\R}-\lambda\big)=\{0\},
\end{equation}
for all $t\in[0,1],\ \lambda<\lambda_{\infty}$. Moreover, if
$\vec{\theta}\not=0$ then 
\begin{equation}
\lb{noegv2f}\begin{split}
\spec\big(-\Delta_{\vec{\theta},\R}+V_{t,\R}\big)\cap(-\infty,0]&=\emptyset
\text{ for all $t\in(0,t_0]$} \\ 
& \text{ provided $t_0$ is sufficiently small.}\end{split}
\end{equation}
\end{lemma}
\begin{proof}
By \cite[Theorem V.4.10]{K80} we have
\begin{equation}
\lb{disspect}
\text{dist} \big(\spec(-\Delta_{\vec{\theta},\R}), \spec(-\Delta_{\vec{\theta},\R}+V_{t,\R})\big)\leq \sup\limits_{t\in[0,1]}\|V_{t,\R}\|_{L^{\infty}(Q)}.
\end{equation}
The inclusion $\spec(-\Delta_{\vec{\theta},\R})\subset [0,\infty)$ and \eqref{disspect} imply \eqref{noegvf} and thus \eqref{cnoegvf}.
We recall that $\spec(-\Delta_{\vec{\theta},\R})=\big\{\|A^{\top}(\vec{\theta}-k)\|^2_{\R^n}\big\}_{k\in\bbZ^n}$ by Theorem \ref{operdef}. Assume $\vec{\theta}\not=0$. Then there exists a positive $\delta$, such that $\min\big(\spec(-\Delta_{\vec{\theta},\R})\big)\geq\delta>0$ since $A^{\top}$ is invertible. 
Next, we observe that
\begin{align}
\|V_{t,\R}\|_{L^{\infty}(Q)}&=t^2\|V(tx)\otimes I_2\|_{L^{\infty}(Q)} \no\\
&\leq t^2 \|V(x)\otimes I_2\|_{L^{\infty}(Q)}\leq \frac{\delta} 2, \text{\ for\ } t<\Big(\frac{\delta}{2\|V(x)\otimes I_2\|_{L^{\infty}(Q)}}\Big)^{\frac 1 2}.\lb{3.13}
\end{align}
By \eqref{3.13}, \eqref{disspect} and $\spec\big(-\Delta_{\vec{\theta},\R}\big)\subset[\delta,+\infty)$ we have
$\spec\big(-\Delta_{\vec{\theta},\R}+V_{t,\R}\big)\subset\big[{\delta}/{2}, +\infty\big)$,
which infers \eqref{noegv2f}.
\end{proof}
\begin{lemma}\lb{constspec}
Let $V_0$ be a symmetric $m\times m$ matrix with real entries and $\alpha\in\R.$ Then
\begin{align}
\spec\big(\alpha(-\Delta_{\vec{\theta}})+V_0\big)=\Big\{\alpha\lambda_k+\mu_l\Big|\lambda_k=\|A^{\top}(\vec{\theta}-k)\|^2_{\R^n},\  k\in\Z^n,&\no\\
\mu_l\in \spec \left(V_0\right),\ 1\leq l\leq m &\Big\}\lb{specsum}
\end{align}
\end{lemma}
\begin{proof} We temporally denote by $-\Delta^{(1)}_{\vec{\theta}}$ the $\vec{\theta}$-periodic Laplacian acting  in $L^2(Q;\C)$. Since $L^2(Q;\C^m)=L^2(Q;\C)\otimes \C^m$, the tensor product of the spaces, we observe that $\alpha(-\Delta_{\vec{\theta}})=\alpha(-\Delta^{(1)}_{\vec{\theta}})\otimes I_m \text{\ and note that $V_0$ is understood as \ } I_{L^2(Q;\C)}\otimes V_0. $ Then, by the abstract result in \cite[Theorem VII.33 (b)]{RS1}, we have
\begin{equation}
\lb{sumdx}
\spec\big(\alpha(-\Delta_{\vec{\theta}})+V_0\big)=\spec\big(\alpha(-\Delta^{(1)}_{\vec{\theta}})\big)+\spec(V_0).
\end{equation} 
Combining \eqref{sumdx} and $\spec(-\Delta^{(1)}_{\vec{\theta}})=\big\{\|A^{\top}(\vec{\theta}-k)\|^2_{\R^n}\big\}_{k\in\bbZ^n}$, we derive \eqref{specsum}.\end{proof} 

We will now apply Theorem \ref{mormas} to compute the Morse index for the Schr\"odinger operator with {\em quasi-periodic} boundary conditions (when $\vec{\theta}\not=0$).
The resulting formula contains the Maslov index of  a path in the space of abstract boundary values that corresponds to the right vertical side of the rectangle $\Gamma$ in Figure 1. In addition, we will express the Morse index via the sum of dimensions of the kernels of the Schr\"odinger operators $-\Delta_{\vec{\theta},t}+V_{tQ}$ obtained by shrinking the unit cell $Q$. 
To define these operators,
for a given $t\in(0,1]$, we set $tQ:=\{y\in \R^n\big| y=tx, x\in Q\}$ and $V_{tQ}:=V\big|_{tQ}$, and consider the rescaled $\vec{\theta}$-periodic real Laplacian
\begin{align}
 -\Delta_{\vec{\theta},\R,t}&:\dom (-\Delta_{\vec{\theta}, \R,t})\subset \Ltwortm\rightarrow \Ltwortm,\no\\
 \dom (-\Delta_{\vec{\theta}, \R,t})&:=\Big\{u \in\Htwortm \Big|\, u^1_j=\cM_{j,t}u^0_{j},\no\\
 &\hskip4cm \partial_ju^1=-\cM_{j,t}\partial_{j}u^0,
  1\leq j\leq n\  \Big\},\no\\
 -\Delta_{\vec{\theta}, \R,t} u&:=-\Delta u, \ u\in  \dom (-\Delta_{\vec{\theta}, \R,t}),\no
\end{align}
where $\cM_{j,t},\ t\in(0,1],\ 1\leq j\leq n$ is the weighted translation operator corresponding to the Dirichlet and Neumann traces for the unit cell $tQ$, cf.\ \eqref{dfncM}. The operator  $-\Delta_{\theta,\R,t}$ is self-adjoint and has compact resolvent. The corresponding complex rescaled $\vec{\theta}$-periodic Laplacianis is defined as follows
\begin{align}
 -\Delta_{\vec{\theta},t}&:\dom (-\Delta_{\vec{\theta},t})\subset \Ltwotm\rightarrow \Ltwotm,\no\\
 \dom (-\Delta_{\vec{\theta},t})&:=\Big\{u \in\Htwotm \Big|\ u^1_j=\bbM_{j,t}u^0_{j},\ \partial_ju^1=-\bbM_{j,t}\partial_{j}u^0,\ 1\leq j\leq n\  \Big\},\no\\
-\Delta_{\vec{\theta},t} u&:=-\Delta u, \ u\in  \dom (-\Delta_{\vec{\theta},t}),\no
\end{align}
where $\bbM_{j,t},\ t\in(0,1],\ 1\leq j\leq n$ is a weighted translation operator corresponding to the Dirichlet and Neumann traces for the unit cell $tQ$, cf.\ \eqref{dfnbM}. The operator  $-\Delta_{\theta,t}$ is self-adjoint and has compact resolvent.

Recall that $\gamma:\dom(-\Delta_{\max})\rightarrow \dom(-\Delta_{\max})/\dom(-\Delta_{\min})$ is the natural projection. Let $\cX_{\vec{\theta}}:=\gamma(\dom(-\Delta_{\vec{\theta},\R}))$.

\begin{theorem}\lb{mormastper}
Assume Hypothesis \ref{phyp} {\it(ii)} and that $\vec{\theta}\not=0$. Then the Morse index of the operator $ -\Delta_{\vec{\theta}}+V$ and the Maslov index of the path $\Upsilon:[\tau,1]\rightarrow F\Lambda\big(\cX_{\vec{\theta}}\big)$ defined by $t\mapsto\gamma\big(\ker(-\Delta_{\max}+V_{t,\R})\big)$ for all $t\in[\tau,1]$ with $\tau\in(0,1]$ sufficiently small are related as follows:
\begin{equation}
\lb{morsmastperf}
2\mo_{\C}\big(-\Delta_{\vec{\theta}}+V\big)=-{\mi \big(\Upsilon,\cX_{\vec{\theta}}\big)}.
\end{equation}
Moreover, if Hypothesis \ref{phyp} (iii) holds then 
\begin{equation}
\lb{mormastperf2}
\mo_{\C}\big(-\Delta_{\vec{\theta}}+V\big)=0,
\end{equation}
if  Hypothesss \ref{phyp} (iv) holds then
\begin{equation}
\lb{mormastperf3}
\mo_{\C}\big(-\Delta_{\vec{\theta}}+V\big)=\sum\limits_{\tau\leq t < 1}\dim_{\C}\big(\ker(-\Delta_{\vec{\theta},t}+V_{tQ})\big).
\end{equation}
Furthermore, if $n=1$ then equality \eqref{mormastperf3} holds assuming that $V(x)$ is a negative definite matrix for all $x\in Q$.
\end{theorem}
\begin{proof}
 Theorem \ref{operdef} and Theorem \ref{bucp} imply that the assumptions of Theorem \ref{mormas} are fulfilled with $A:=-\Delta_{\min}$, $A_\cD:=-\Delta_{\theta,\R}$ and $V_t:=V_{t,\R}$. Thus, for sufficiently small $\tau>0$ we obtain 
 \begin{align}
\mo_{\R}(-\Delta_{\vec{\theta},\R}&+V_{\tau,\R})-\mo_{\R}(-\Delta_{\vec{\theta},\R}+V_{1,\R})\no\\
&=\mi(\gamma(\ker(-\Delta_{\max}+V_{t,\R}))|_{\tau\leq t\leq 1},\cX_{\vec{\theta}}). \lb{3.47}
 \end{align}
 By \eqref{noegv2f} we have
 \begin{equation}
 \lb{3.50}
 \mo_{\R}(-\Delta_{\theta,\R}+V_{\tau,\R})=0.
 \end{equation}
 It follows from Proposition \ref{3.9} that 
 \begin{align}
& \mo_{\R}(-\Delta_{\theta,\R}+V_{\tau,\R})=2\mo_{\C}(-\Delta_{\theta}+V_{\tau}),\lb{3.48}\\
& \mo_{\R}(-\Delta_{\theta,\R}+V_{1,\R})=2\mo_{\C}(-\Delta_{\theta}+V),\lb{3.49}
 \end{align}
and thus \eqref{morsmastperf} holds. If Hypothesis \ref{phyp} {\it (iii)} holds, then the family of operators of multiplication by $V_{t,\R}(x)$  satisfies the assumptions of Theorem \ref{mormas} {\it (ii)}, implying
 \begin{align}
& \mi \big(\Upsilon,\cX_{\vec{\theta}}\big)=\sum\limits_{\tau< t \leq 1}\dim_{\R}\big(\ker(-\Delta_{\theta,\R}+V_{t,\R})\big)\lb{mascalc},\\
& \mo_{\C}\big(-\Delta_{\vec{\theta}}+V\big)=-\frac 1 2\sum\limits_{\tau< t \leq 1}\dim_{\R}\big(\ker(-\Delta_{\vec{\theta},\R}+V_{t,\R})\big)\lb{mascalc2}.
 \end{align}
Since the left hand side of \eqref{mascalc2} is non-negative and the right hand side is non-positive we derive \eqref{mormastperf2}.
 
 Similarly, assuming Hypothesis \ref{phyp} {\it (iii)} we obtain
 \begin{align}
 \mo_{\C}\big(-\Delta_{\vec{\theta}}+V\big)&=\frac 1 2\sum\limits_{\tau< t \leq 1}\dim_{\R}\big(\ker(-\Delta_{\vec{\theta},\R}+V_{t,\R})\big)\lb{3.53}\\
& =\sum\limits_{\tau\leq t < 1}\dim_{\C}\big(\ker(-\Delta_{\vec{\theta}}+V_{t})\big).\lb{3.54}
\end{align}
Let $t_*\in[\tau,1)$ be a crossing, i.e. $$\dim_{\R}\big(\ker(-\Delta_{\vec{\theta},\R}+V_{t_*,\R})\big)=2\dim_{\C}\big(\ker(-\Delta_{\vec{\theta}}+V_{t_*})\big)~\not=~0,$$
that is, there exists $0\not=u\in\Htwom$ such that
\begin{equation}\lb{3.61a}
\left\{
\begin{array}{lcl}
-\Delta u(x) +t_*^2V(t_*x) u(x)=0,\ \text{in}\ \Ltwom,& &  \\
u^1_j=\bbM_{j}u^0_{j},\ \partial_ju^1=-\bbM_{j}\partial_{j}u^0,\ 1\leq j\leq n.       &  &  \\
\end{array}
\right.
\end{equation}
Changing variables $y:=t_*x$, one obtains the equivalent boundary value problem 
\begin{equation}\lb{3.61b}
\left\{
\begin{array}{lcl}
-\Delta u(y) +V_{t_*Q}(y) u(y)=0,\ \text{in}\ L^2(t_*Q, \C^m),& &  \\
u^1_j=\bbM_{j,t_*}u^0_{j},\ \partial_ju^1=-\bbM_{j,t_*}\partial_{j}u^0,\ 1\leq j\leq n,      &  &  \\
\end{array}
\right.
\end{equation}
and derives the inequality $\dim_{\C}\big(\ker(-\Delta_{\vec{\theta}}+V_{t_*})\big)\leq \dim_{\C}\big(\ker(-\Delta_{\vec{\theta},t}+V_{tQ})\big)$. 
Conversely, the change of variables $x:=t_*^{-1}y$ in \eqref{3.61b} implies \eqref{3.61a} and $\dim_{\C}\big(\ker(-\Delta_{\theta,t}+V_{tQ})\big)\leq \dim_{\C}\big(\ker(-\Delta_{\theta}+V_{t_*})\big)$ and, thus, \eqref{mormastperf3} holds. 

The last assertion in the theorem  follows from Theorem \ref{onedim} proved below.
\end{proof}

We are now in the position to apply Theorem \ref{mormas} to compute  the Morse index of the {\em periodic} Schr\"{o}dinger operator (i.e., when $\vec{\theta}=0$). The resulting formula contains the Maslov index of the flow corresponding to the right vertical side of the rectangle $\Gamma$ in Figure 1 and, in addition, a term corresponding to the 
to the lower horizontal part of the
rectangle. We begin with a computation of the Morse index of the operator corresponding lower horizontal part  in Figure 1.

\begin{proposition}
\lb{b1per} Assume $\vec{\theta}=0$, Hypothesis \ref{phyp} {\it(i)} and that $V$ is continuous at $0$, and the matrix $V(0)$ is invertible. Then for a sufficiently small $\tau\in(0,1)$ one has 
\begin{align}
&(i)\ 0\not\in\spec\big(-\Delta_{0,\R}+V_{\tau,\R}\big),\lb{b1perf1}\\
&(ii)\ \mo_{\R}\big(-\Delta_{0,\R}+V_{\tau,\R}\big)=2\mo \big(V(0)\big).\lb{b1perf2}
\end{align}
\end{proposition}
\begin{proof}
By Propositiom \ref{3.9}, proving assertions \eqref{b1perf1}, \eqref{b1perf2} is equivalent to showing similar results for the complex periodic Shr\"{o}dinger operator, that is, to show that
\begin{align}
&(i')\ 0\not\in\spec\big(-\Delta_{0}+V_{\tau}\big),\no\\
&(ii')\ \mo_{\C}\big(-\Delta_{0}+V_{\tau}\big)=\mo \big(V(0)\big).\no
\end{align}
Moreover, since $0\not\in\spec\big(-\Delta_{0}+V_{\tau}\big)$ if and only if $0\not\in\spec\big(\tau^{-2}(-\Delta_{0})+\tau^{-2}V_{\tau}\big)$ and 
$\mo\big(\tau^{-2}(-\Delta_{0})+\tau^{-2}V_{\tau}\big)=\mo\big(-\Delta_{0}+V_{\tau}\big)$ for $\tau\in(0,1)$, it is enough to show
\begin{align}
	&(i'')\ 0\not\in\spec\big(\tau^{-2}(-\Delta_{0})+\tau^{-2}V_{\tau}\big),\lb{b1perf11}\\
	&(ii'')\ \mo_{\C}\big(\tau^{-2}(-\Delta_{0})+\tau^{-2}V_{\tau}\big)=\mo \big(V(0)\big),\lb{b1perf22}
\end{align}
for some $\tau\in(0,1)$.

By \cite[Theorem V.4.10]{K80} and the assumption on  continuity of $V$ we infer
\begin{align}
\dist&\big(\spec(\tau^{-2}(-\Delta_{0})+\tau^{-2}V_{\tau}),\spec( \tau^{-2}(-\Delta_{0})+V(0))\big)\no\\
&\leq \|V(\tau x)-V(0)\|_{L^{\infty}(Q)}\rightarrow 0,\text{\ as\ } \tau\rightarrow 0. \lb{3.19} 
\end{align} 
By Lemma \ref{constspec} with $\alpha=\tau^{-2}$ and $V_0=V(0)$ we have
\begin{align}
\spec\big(\tau^{-2}(-\Delta_{0})+V(0) \big)=\Big\{\tau^{-2}\lambda_k+\mu_j\Big|\lambda_k=\|A^{\top}k\|^2_{\R^n},\  k\in\Z^n,& \lb{3.22}\\
\mu_j\in \spec \left(V(0)\right),\ 1\leq j\leq m &\Big\}\no.
\end{align}
Since $V(0)$ is invertible, there exists $\delta>0$, such that the eigenvalues $\{\mu_1, \cdots, \mu_m\}=\spec\big(V(0)\big)$ can be ordered as follows,
\begin{align}
-\|V(0)\|_{\R^{m\times m}}&\leq \mu_1\leq\dots\leq \mu_{\mo(V(0))}<-\delta\no\\
&<0<\delta< \mu_{\mo(V(0))+1}\leq \cdots\leq \mu_m\leq \|V(0)\|_{\R^{m\times m}}.\lb{3.23}
\end{align}
By \eqref{3.22} and \eqref{3.23} for $0<\tau<2\pi\min\big\{\|A^Tk\|\delta^{-1/2}\big|k\in\Z^n\setminus\{0\}\big\}$ one has
\begin{equation}\lb{3.24}
\spec\big(\tau^{-2}(-\Delta_{0})+V(0)\big)\subset\big[-\|V(0)\|_{\R^{m\times m}}-\delta,-\delta\big]\cup\big[\delta,+\infty\big).
\end{equation}
In addition, for a positive $\tau_0<2\pi\min\big\{\|A^Tk\|\delta^{-1/2}\big|k\in\Z^n\setminus\{0\}\big\}$ so small that $\|V(\tau x)-V(0)\|_{L^{\infty}(Q)}<\delta/2 \text{\ for all\ }\tau<\tau_0$, from \eqref{3.19} and \eqref{3.24} one obtains
\begin{equation}\lb{3.25}
\spec\big(\tau^{-2}(-\Delta_{0})+\tau^{-2}V_{\tau}\big)\subset\big[-\|V(0)\|_{\R^{m\times m}}-\delta/2,-\delta/2\big]\cup\big[\delta/2,+\infty\big),
\end{equation}
implying \eqref{b1perf11} and \eqref{b1perf1}.

Next, let us consider the Riesz projections defined as
\begin{align}
&P_{\tau}:=\frac{1}{2\pi\bfi}\oint\limits_{C} \big(z-\tau^{-2}(-\Delta_{0})-\tau^{-2}V_{\tau} \big)^{-1}dz,\ \ \tau\in(0,\tau_0],\lb{rpr}\\
&P^{0}_{\tau}:=\frac{1}{2\pi\bfi}\oint\limits_{C} \big(z-\tau^{-2}(-\Delta_{0})-V(0) \big)^{-1}dz,\ \ \tau\in(0,\tau_0],
\end{align} 
where the contour $C$ is the rectangle with the vertices at the points $\pm\bfi\delta$ and $-\|V(0)\|_{\R^{m\times m}}-\delta\pm\bfi\delta$. Observe that for all $\tau\in(0,\tau_0]$ and $z\in\bbC$ we have
\begin{align}
\|(z-\tau^{-2}(-\Delta_{0})&-\tau^{-2}V_{\tau} \big)^{-1}\|_{\cB(\Ltwom)}\no\\&\leq\big(\dist\big(z,\spec(\tau^{-2}(-\Delta_{0})-\tau^{-2}V_{\tau}) \big)\big)^{-1}\lb{resnorm}\\
&\leq\big(\dist\big(C,\spec(\tau^{-2}(-\Delta_{0})-\tau^{-2}V_{\tau}) \big)\big)^{-1}\leq2/\delta,\no
\end{align}
and similarly 
\begin{equation}\lb{3.27}
\|(z-\tau^{-2}(-\Delta_{0})-V(0) \big)^{-1}\|_{\cB(\Ltwom)}\leq2/\delta.
\end{equation}
Then
\begin{align}
P_{\tau}-P_{\tau}^0 =\frac{1}{2\pi\bfi}\oint\limits_{C}\big(z-\tau^{-2}(-\Delta_{0})-&\tau^{-2}V_{\tau} \big)^{-1}\big(\tau^{-2}V_{\tau}-V(0)\big)\no\\
&\times(z-\tau^{-2}(-\Delta_{0})-V(0) \big)^{-1}dz.\lb{prdif}
\end{align}
From the continuity assumption on $V$ and \eqref{resnorm},\,\eqref{3.27},\, \eqref{prdif}  we obtain 
\begin{equation}
\lb{3.30}
\lim\limits_{\tau\rightarrow 0}\|P_{\tau}-P_{\tau}^0\|_{\cB(\Ltwom)}=0.
\end{equation}
To prove \eqref{b1perf2}, we claim that
\begin{equation}
\lb{morsfin}
c_0:=\sup\big\{ \mo\big(\tau^{-2}(-\Delta_{0})+\tau^{-2}V_{\tau} \big)\big|\, \tau\in (0,\tau_0]\big\}<+\infty.
\end{equation}
Assuming the claim, we note that $\tr P_{\tau}=\dim\ran(P_{\tau})=\mo(\tau^{-2}(-\Delta_{0})+\tau^{-2}V_{\tau})$ and $\tr P_{\tau}^0=\mo(V(0))$. Next, using the claim, we estimate
\begin{align}
|\tr P_{\tau}- \tr P_{\tau}^0|&=|\tr(P_{\tau}-P_{\tau}^0)|\no\\
&\leq \big(\dim\ran(P_{\tau})+\dim\ran(P_{\tau}^0)\big)\|P_{\tau}-P_{\tau}^0\|_{\cB(\Ltwom)}\no\\
&\leq\big(c_0+\mo(V(0))\big)\|P_{\tau}-P_{\tau}^0\|_{\cB(\Ltwom)}\rightarrow 0, \text{as} \ \tau\rightarrow 0.
\end{align}
Since the functions $\tau\mapsto\tr P_{\tau}$ and $\tau\mapsto\tr P_{\tau}^0$ take integer values we conclude that $\tr P_{\tau}=\tr P_{\tau}^0$ for all sufficiently small $\tau\in(0,1)$, implying the desired equality $\mo(\tau^{-2}(-\Delta_{0})+\tau^{-2}V_{\tau})=\mo(V(0))$.

We conclude the proof by showing claim \eqref{morsfin}. Fix any $\tau\in(0,\tau_0]$. First, since $$2\mo_{\C}\big(\tau^{-2}(-\Delta_{0})+\tau^{-2}V_{\tau} \big)=2\mo_{\C}\big(-\Delta_{0}+V_{\tau} \big)=\mo_{\R}\big(-\Delta_{0,\R}+V_{\tau,\R} \big),$$ it is enough to show
that $\sup\big\{ \mo_{\R}\big(-\Delta_{0,\R}+V_{\tau,\R} \big)\big| \tau\in (0,\tau_0]\big\}<+\infty$. Second, \eqref{3.25} and Proposition \ref{3.9} imply 
\begin{equation}\lb{3.34}
0\not \in \spec\big(-\Delta_{0,\R}+V_{\tau,\R}\big)\ \text{for all}\ \tau\in(0,\tau_0].
\end{equation}
Theorem \ref{operdef}, Lemma \ref{noegv} and Theorem \ref{bucp} imply the assumptions of Theorem \ref{mormas}  with $A_\cD=-\Delta_{0,\R}$ and $V_t=V_{t,\R}$. Using Theorem \ref{mormas} we therefore infer
\begin{align}
\mo_{\R}(-\Delta_{0,\R}&+V_{\tau,\R})-\mo_{\R}(-\Delta_{0,\R}+V_{1,\R})\no\\&=\mi(\gamma(\ker(-\Delta_{\max})+V_{t,\R})|_{\tau\leq t\leq 1},\gamma(\dom(-\Delta_{0,\R}))).\lb{3.33}
\end{align}
By  \eqref{3.34}, the path $t\mapsto \gamma(\ker(-\Delta_{\max})+V_{t,\R})|_{\tau\leq t\leq 1}$ does not have any crossings on the interval $t\in(0,\tau_0]$, thus
\begin{align}
\mi(\gamma&(\ker(-\Delta_{\max})+V_{t,\R})|_{\tau\leq t\leq 1},\gamma(\dom(-\Delta_{0,\R})))\no\\&=\mi(\gamma(\ker(-\Delta_{\max})+V_{t,\R})|_{\tau_0\leq t\leq 1},\gamma(\dom(-\Delta_{0,\R})))\lb{nocros}  
\end{align}
for all $\tau\leq \tau_0$.
Finally, \eqref{3.33} and \eqref{nocros} imply
\begin{align}
\mo_{\R}&(-\Delta_{0,\R}+V_{\tau,\R})=\mo_{\R}(-\Delta_{0,\R}+V_{1,\R})\no\\
&+\mi(\gamma(\ker(-\Delta_{\max})+V_{t,\R})|_{\tau_0\leq t\leq 1},\gamma(\dom(-\Delta_{0,\R}))),\lb{3.36}
\end{align}
and \eqref{morsfin} follows from  \eqref{3.36}.
\end{proof}

We are ready to formulate our principal result for the periodic case.

\begin{theorem}\lb{mormasper}
Assume Hypotheses \ref{phyp} (ii), (v) and that $\vec{\theta}=0$. Then the Morse index of the operator $ -\Delta_0+V$ and the Maslov index of the path $\Upsilon:[\tau,1]\rightarrow F\Lambda\big(\cX_{0}\big)$ defined by $t\mapsto\gamma\big(\ker\big(-\Delta_{\max}+V_{t,\R}\big)\big)$ for all $t\in[\tau,1]$ with $\tau\in(0,1]$ sufficiently small are related as follows:
\begin{equation}
\lb{mormasperf}
2\mo_{\C}\big(-\Delta_0+V\big)-2\mo\big(V(0)\big)=-{\mi \big(\Upsilon,\cX_{0}\big)}.
\end{equation}
Moreover, if Hypotheses \ref{phyp} (iii),(v) hold then 
\begin{equation}
\lb{mormasperf2}
\mo_{\C}\big(-\Delta_0+V\big)-\mo\big(V(0)\big)=-\sum\limits_{\tau< t \leq 1}\dim_{\C}\big(\ker(-\Delta_{0,t}+V_{tQ})\big),
\end{equation}
if  Hypotheses \ref{phyp} (iv),(v) hold then
\begin{equation}
\lb{mormasperf3}
\mo_{\C}\big(-\Delta_0+V\big)-\mo\big(V(0)\big)=\sum\limits_{\tau\leq t < 1}\dim_{\C}\big(\ker(-\Delta_{0,t}+V_{tQ})\big).
\end{equation}
\end{theorem}
\begin{proof}
Employing \eqref{b1perf2} and Theorem \ref{mormas}, one repeats the proof of Theorem \ref{mormastper}  and derives assertions
 \eqref{mormasperf}, \eqref{mormasperf2} and \eqref{mormasperf3}.
\end{proof}
\begin{remark}\lb{rmk3.16}
	In fact, equalities  \eqref{mormastperf2} and \eqref{mormasperf2} (correspondingly, \eqref{mormastperf3} and \eqref{mormasperf3}) hold assuming that Hypothesis \ref{phyp} {\it (iii)} is replaced by the following weaker assumption: $V\in C^1(Q;\R^{m\times m})$  and $\min\spec\big(2tV(tx)+t^2\nabla V(tx)x\big)>0$ (correspondingly $\min\spec\big(2tV(tx)+t^2\nabla V(tx)x\big)<0$) for each conjugate point $t_*\in(0,1]$. In other words, it is enough to assume that operator of multiplication by $2t_*V(t_*x)+t_*^2\nabla V(t_*x)x$ is positive (correspondingly negative) definite on the finite dimensional space $\ker\big(-\Delta_{\vec{\theta}}+V_{t_*}\big)$, for those $t_*$ where it is non-trivial.  
\end{remark}

The computation of the Maslov index in the space of boundary values looks rather abstract. The next theorem contains a fairly explicit formula for the Maslov form at crossings in terms of the potential $V$. This computation is of independent interest and is  used in Theorem \ref{onedim}.
\begin{theorem}\lb{3.19teor}
	Assume Hypothesis \ref{phyp} {\it (ii)}. Then for $u_{\R}:=(u_1,\dots, u_{2m})\in\ker(-\Delta_{\vec{\theta},\R}+V_{t_*,\R}(x))$ one has
	\begin{align}
	&\cQ_{t_*,\cX_{\vec{\theta}}}([u_{\R}],[u_{\R}])=\int\limits_{Q}\Big(u_{\R}(x),\frac{\partial(V_{t,\R})}{\partial t}\Big|_{{t={t_{*}}}}u_{\R}(x)\Big)_{\R^{2m}}d^nx\lb{t3.21}\\
	&\quad= t_*\Re\int\limits_{\partial Q}(x,\vec{\nu})_{\R^n}\big({\bf u}(x),V(t_*x){\bf u}(x)\big)_{\C^m}d^{n-1}x\no\\
	&\qquad+t_*^{-1}\int\limits_{\partial Q}(x,\vec{\nu})_{\R^n}\|\nabla {\bf u}(x)\|_{\C^n}^2 d^{n-1}x-2t_*^{-1}\Re\int\limits_{\partial Q}\big(\nabla {\bf u}(x) x, \nabla {\bf u}(x)\vec{\nu}\big)_{\C^m}d^{n-1}x,\no
	\end{align}
	where ${\bf u}(x)=(u_1(x)+\bfi u_2(x),\dots, u_{2m-1}(x)+\bfi u_{2m}(x))\in \C^m.$
\end{theorem}
\begin{proof}
	The calculation of the form in Lemma \ref{g2} yields the first equality in \eqref{t3.21}. In order to show the second equality, we notice that ${\bf u}\in\ker(-\Delta_{\vec{\theta}}+t_{*}^2V(t_{*}))$, i.e. ${\bf u}\in\dom(-\Delta_{\vec{\theta}})$ and $-\Delta {\bf u}(x)+t_{*}^2V(t_{*}x){\bf u}(x)=0,$ a.e. $x\in Q$, that is
	\begin{align}
	& t_{*}^{-2}\Delta {\bf u}(x/t_{*})=V(x){\bf u}(x/t_{*}), \text{\ a.e\ } x\in t_*Q,\lb{weaksol1}
	\end{align}
	and, since ${\bf u}$ satisfies the $\vec{\theta}$-periodic boundary conditions we have
	\begin{align}
	&\int\limits_{\partial Q}\big({\bf u}(x),{\nabla {\bf u}(x)}\vec{\nu}\big)_{\C^m}d^{n-1}x=0.\lb{bound2}   
	\end{align}
	We will denote by $\nabla V(x)x$ the $m\times m$ matrix $\{(\nabla V_{kl}(x))x\}_{1\leq k,l\leq m}$, and  will use notation $\text{div}(V(x)x)$ for $\{\text{div}(V_{kl}(x)x)\}_{1\leq k,l\leq m}$. Then we have 
	\begin{align}
	\int\limits_{Q}\Big({\bf u}(x)&,\left(2t_*V(t_*x)+t_*^2\nabla V(t_*x)x\right){\bf u}(x)\Big)_{\C^m}dx\no\\
	&=t_*^{1-n}\int\limits_{t_*Q}\Big({\bf u}\Big(\frac y {t_*}\Big),\left(2V(y)+\nabla V(y)y\right){\bf u}\Big(\frac y {t_*}\Big)\Big)_{\C^m}d^ny\no\\
	&=t_*^{1-n}(2-n)\int\limits_{t_*Q}\Big({\bf u}\Big(\frac y {t_*}\Big),V(y){\bf u}\Big(\frac y{t_*}\Big)\Big)_{\C^m}d^ny\no\\
	&+t_*^{1-n}\int\limits_{t_*Q}\Big({\bf u}\Big(\frac y {t_*}\Big),\text{div}(V(y)y){\bf u}\Big(\frac y {t_*}\Big)\Big)_{\C^m}d^ny=:t_*^{1-n}(I+II).
	\end{align}
	Using \eqref{weaksol1} one obtains, by Green's formula and \eqref{bound2}, the equality 
	\begin{align}
	I&=\frac{2-n}{t_*^2}\int\limits_{t_*Q}\Big({\bf u}\Big(\frac y {t_*}\Big),\Delta {\bf u}\Big(\frac y{t_*}\Big)\Big)_{\C^m}d^ny=
	(2-n)t_*^{n-2}\int\limits_{Q}\big({\bf u}(x),\Delta {\bf u}(x)\big)_{\C^m}d^nx\no\\
	&=(2-n)t_*^{n-2}\Big(\int\limits_{\partial Q}\big({\bf u}(x),\nabla {\bf u}(x)\vec{\nu}\big)_{\C^m}d^{n-1}x-
	\int\limits_{Q}\big(\nabla {\bf u}(x),\nabla {\bf u}(x)\big)_{\C^m}d^nx\Big)\no\\
	&=(n-2)t_*^{n-2}\int\limits_{Q}\big(\nabla {\bf u}(x),\nabla {\bf u}(x)\big)_{\C^m}d^nx\lb{I}.
	\end{align}
Similarly, using the divergence theorem, \eqref{weaksol1} and Green's formula, we compute:
	\begin{align}
	& II=\sum\limits_{1\leq k,l\leq m} \ \int\limits_{{t_*}Q} {\bf u}_k\Big(\frac y {t_*}\Big)\overline{{\bf u}_l\Big(\frac y {t_*}\Big)}\text{div}(V_{kl}(y)y)d^ny\no\\
	&=t_{*}^{n}\Big(\int\limits_{\partial Q}(x,\vec{\nu})_{\R^n}\big({\bf u}(x),V(t_*x){\bf u}(x)\big)_{\C^m}d^{n-1}x\no\\&\hspace{13em}-
	2\Re\int\limits_{Q}\big(\nabla {\bf u}(x) x, V(t_*x){\bf u}(x)\big)_{\C^m}d^nx\Big)\no\\
	&=t_{*}^{n}\Big(\int\limits_{\partial Q}(x,\vec{\nu})_{\R^n}\big({\bf u}(x),V(t_*x){\bf u}(x)\big)_{\C^m}d^{n-1}x\no\\&\hspace{13em}-
	2t_*^{-2}\Re\int\limits_{Q}\big(\nabla {\bf u}(x) x, \Delta {\bf u}(x)\big)_{\C^m}d^nx\Big)\lb{IIh}\\
	&=t_{*}^{n}\Big(\int\limits_{\partial Q}(x,\vec{\nu})_{\R^n}\big({\bf u}(x),V(t_*x){\bf u}(x)\big)_{\C^m}d^{n-1}x\no\\&\hspace{13em}-
	2t_*^{-2}\Re\int\limits_{\partial Q}\big(\nabla {\bf u}(x) x, \nabla {\bf u}(x)\vec{\nu}\big)_{\C^m}d^{n-1}x\no\\
	&+2t_*^{-2}\Re\sum_{i=1}^{m}\int\limits_{Q}\big(\nabla (\nabla {\bf u}_i(x) x), \nabla {\bf u}_i(x)\big)_{\C^n}d^nx\Big)\no\\
	&=t_*^{n}\Big(\int\limits_{\partial Q}(x,\vec{\nu})_{\R^n}\big({\bf u}(x),V(t_*x){\bf u}(x)\big)_{\C^m}d^{n-1}x\no\\&\hspace{13em}
	-2t_*^{-2}\Re\int\limits_{\partial Q}\big(\nabla {\bf u}(x) x, \nabla {\bf u}(x)\vec{\nu}\big)_{\C^m}d^{n-1}x\no\\
	&+t_*^{-2}\sum_{i=1}^{m}\int\limits_{Q}x\nabla(\|\nabla {\bf u}_i(x)\|_{\C^n}^2)+2\|\nabla {\bf u}_i(x)\|^2_{\C^n}\ d^nx\Big)=\lb{III}\\
	&=t_{*}^{n}\Big(\int\limits_{\partial Q}(x,\vec{\nu})_{\R^n}\big({\bf u}(x),V(t_*x){\bf u}(x)\big)_{\C^m}d^{n-1}x\no\\&\hspace{0em}
	+2t_*^{-2}\Re\int\limits_{\partial Q}\big(\nabla {\bf u}(x) x, \nabla {\bf u}(x)\vec{\nu}\big)_{\C^m}d^{n-1}x+(2-n)t_*^{-2}\sum_{i=1}^{m}\int\limits_{Q}\|\nabla {\bf u}_i(x)\|_{\C^n}^2\ d^nx\no\\
	& \hspace{13em}+	t_*^{-2}\sum_{i=1}^{m}\int\limits_{\partial Q}(x,\vec{\nu})_{\R^n}\|\nabla {\bf u}_i(x)\|_{\C^n}^2 d^{n-1}x\Big).\lb{II}
	\end{align}
	In \eqref{III} we used the relation 
	$$\nabla (\nabla \overline{v\left(x\right)}x)\nabla v(x)+\nabla (\nabla {v\left(x\right)}x)\nabla \overline{v(x)}=x\nabla(|\nabla v(x)|^2)+2|\nabla v(x)|^2,$$
	for all $v\in H^2(Q)$.
	Combining \eqref{I} and \eqref{II} with
	\begin{align}
	\int\limits_{Q}\Big(&u_{\R}(x),\frac{\partial(V_{t,\R}(x))}{\partial t}\Big|_{{t={t_{*}}}}u_{\R}(x)\Big)_{\R^{2m}}d^nx\no\\&
	=\Re\int\limits_{Q}\Big({\bf u}(x),\frac{\partial(V_t(x))}{\partial t}\Big|_{{t={t_{*}}}}{\bf u}(x)\Big)_{\C^{m}}d^nx,
	\end{align}
	one obtains \eqref{t3.21}.
\end{proof}
\begin{remark}
Assuming the hypotheses in Theorem \ref{3.19teor}, and that the vectors $a_j$ from \eqref{dfnL} satisfy $(a_i, a_j)_{\R^n}=0$ for $i\not=j$, one has
\begin{align}
&\cQ_{t_*,\cX_{\vec{\theta}}}([u_{\R}],[u_{\R}])=\int\limits_{Q}\Big(u_{\R}(x),\frac{\partial(V_{t,\R}(x))}{\partial t}\Big|_{{t={t_{*}}}}u_{\R}(x)\Big)_{\R^{2m}}d^nx\lb{r3.20}\\
&= t_*\sum_{k=1}^{n}\Re\int\limits_{{\partial Q}_k^0}(x,\vec{\nu})_{\R^n}\big({\bf u}(x),V(t_*x)+V(t_*(x+a_k)){\bf u}(x)\big)_{\C^{m}}d^{n-1}x\no\\
&+t_*^{-1}\int\limits_{\partial Q}(x,\vec{\nu})_{\R^n}\sum_{i=1}^{m}\|\nabla {\bf u}_i(x)\|_{{\R^{n}}}^2 d^{n-1}x-2t_*^{-1}\Re\int\limits_{\partial Q}\big(\nabla {\bf u}(x) x, \nabla {\bf u}(x)\vec{\nu}\big)_{\C^{m}}d^{n-1}x.\no
\end{align}
\end{remark}
In the one dimensional case we recover a result from \cite{JLM}.
\begin{theorem}\lb{onedim}
	Assume Hypothesis \ref{phyp} {\it (ii)} and that $n=1$. Suppose that $t_*$ is a crossing. If the matrix $V(t_*a_1)$ is negative definite, then for all $u_{\R}:=(u_1,\dots, u_{2m})\in\ker\big(\big(-\frac{d^2}{dx^2}\big)_{\vec{\theta},\R}+V_{t_{*},\R}\big)$ one has
	\begin{align}
	&\cQ_{t_*,\cX_{\vec{\theta}}}([u_{\R}],[u_{\R}])=\int\limits_{Q}\big(u_{\R}(x),\frac{\partial(V_{t,\R}(x))}{\partial t}\big|_{{t={t_{*}}}}u_{\R}(x)\big)_{\R^{2m}}dx<0.\lb{3.60}
	\end{align}
	That is, the crossings of the path $t\mapsto\gamma\big(\ker\big(-\partial_x^2\big)_{\max}+V_{t,\R}\big)$ as $t\in[\tau,1]$ with any $\tau>0$, are negative. Therefore, assertion \eqref{mormastperf3} holds if $\vec{\theta}\not=0$, and \eqref{mormasperf3} holds if $\vec{\theta}=0$ and $V(0)$ is invertible. 
\end{theorem}
\begin{proof}
	By \eqref{t3.21} with $n=1$ we have
	\begin{align}
	\int\limits_{Q}&\Big(u_{\R}(x),\frac{\partial(V_{t,\R}(x))}{\partial t}\Big|_{{t={t_{*}}}}u_{\R}(x)\Big)_{2m}dx\no\\
	&= t_*\Re\big({\bf u}(t_*a_1),V(t_*a_1){\bf u}(t_*a_1)\big)_{\C^m}-t_*^{-1}\sum\limits_{i=1}^{m}|{\bf u}_i'(t_*a_1)|^2<0\no.
	\end{align}
	The rest follows from Theorems \ref{mormastper} and \ref{mormasper}.
\end{proof}

\subsection{Schr\"{o}dinger operators with Dirichlet and Neumann boundary conditions} As yet another application of the abstract results from Section \ref{sec:2}, in this subsection  we give an alternative proof of the celebrated Morse-Smale Index Theorem, and its analogue for the Neumann boundary conditions recently discussed in \cite{CJLS}. Let $-\Delta_{D,\Omega}$ denote the Dirichlet Laplacian.
\begin{theorem}\lb{dir}
	Let $\Omega$ be open bounded star-shaped Lipschitz domain in $\R^n$ and assume that $V\in C^1(\R^n, \R^{m\times m})$ and $V(x)=V(x)^{\top}$, $x\in\Omega$. Denote $\cX_D=\gamma(\dom\big(-\Delta_{D,\Omega})\big)$. Then the Morse index of the Schr\"{o}dinger operator with the Dirichlet boundary conditions, 
	\begin{align}L_{D}u&:=-\Delta_{D,\Omega}u+Vu,\no\\  \dom(-\Delta_{D,\Omega})&=\{u\in H^1(\Omega,\C^m)\big| \Delta\in L^2(\Omega, \C^m),\ \gamma_{D}u=0 \text{\ in\ } H^{1/2}(\Omega,\C^m)\}\no,\end{align} 
	and,  for $\tau\in(0,1]$ sufficiently small, the Maslov index of the path $\Upsilon:[\tau,1]\rightarrow F\Lambda(\cX_D)$, defined by $t\mapsto\gamma\big(\ker\big(-\Delta_{\max}+V_{t,\R}\big)\big)$ for all $t\in[\tau,1]$ are related as follows:
	\begin{equation}
	\lb{dirf1}
	2\mo_{\C}\big(-\Delta_{D,\Omega}+V\big)=-{\mi \big(\Upsilon,\cX_D\big)}.
	\end{equation}
	If, in addition, $\Omega$ has $C^{1,r}$-boundary for some $1/2<r<1$ and $n\geq 2$ then 
	\begin{equation}
	\lb{dirf2}
	\mo_{\C}\big(-\Delta_{D,\Omega}+V\big)=\sum\limits_{\tau\leq t < 1}\dim_{\C}\big(\ker(-\Delta_{D,\Omega}+t^2V(tx))\big).
	\end{equation}
\end{theorem}
\begin{proof}
Theorem \ref{mormas} {\it (i)} with $A:=-\Delta_{\min}$ and $A_\cD:=-\Delta_{D,\Omega}$ imply \eqref{dirf1}. Theorem \ref{mormas} {\it (iii)} together with Corollary 5.7 in \cite{CJLS} imply \eqref{dirf2}.
\end{proof}
Similar proofs work for the Neumann Laplacian $-\Delta_{N,\Om}$.
\begin{theorem}\lb{dirn}
	Let $\Omega$ be an open bounded star-shaped Lipschitz domain in $\R^n$ and assume that $V\in C^1(\R^n, \R^{m\times m})$ and  $V(x)=V(x)^{\top}$, $x\in\Omega$. Assume that the matrix $V(0)$ is invertible and denote $\cX_N=\gamma(\dom\big(-\Delta_{N,\Omega})\big)$. Then the Morse index of the Schr\"{o}dinger operator with the Neumann boundary conditions,  
	\begin{align}L_{\cN}u&:=-\Delta_{N,\Omega}u+Vu,\no\\  \dom(-\Delta_{N,\Omega})&=\{u\in H^1(\Omega,\C^m)\big| \Delta\in L^2(\Omega, \C^m),\ \gamma_{N}u=0 \text{\ in\ } H^{-1/2}(\Omega,\C^m)\}\no,\end{align} 
	and, for $\tau\in(0,1]$ sufficiently small, the Maslov index of the path $\Upsilon:[\tau,1]\rightarrow F\Lambda(\cX_N)$ defined by $t\mapsto\gamma\big(\ker\big(-\Delta_{\max}+V_{t,\R}\big)\big)$ for all $t\in[\tau,1]$ are related as follows:
	\begin{equation}\no
	2\mo_{\C}\big(-\Delta_{N,\Omega}+V\big)=-{\mi \big(\Upsilon,\cX_N\big)}+2\mo(M(0)).
	\end{equation}
	Moreover, if $\min\spec\big(2tV(tx)+t^2\nabla V(tx)x\big)>0$ for each $t\in(0,1]$ and almost all $x\in \Omega$, then for a sufficiently small $\tau>0$ one has
	\begin{equation}\no
	\mo_{\C}\big(-\Delta_{N,\Omega}+V\big)=-\sum\limits_{\tau< t \leq 1}\dim_{\C}\big(\ker(-\Delta_{N,\Omega}+t^2V(tx))\big)+\mo(M(0)), 
	\end{equation}
	 while if $\min\spec\big(2tV(tx)+t^2\nabla V(tx)x\big)<0$ for each $t\in(0,1]$ and almost all $x\in \Omega$ then for a sufficiently small $\tau>0$ one has
	\begin{equation}\no
	\mo_{\C}\big(-\Delta_{N,\Omega}+V\big)=\sum\limits_{\tau\leq t < 1}\dim_{\C}\big(\ker(-\Delta_{N,\Omega}+t^2V(tx))\big)+\mo(M(0)).
	\end{equation}
\end{theorem}

\section{The symplectomorphism between the space of abstract boundary values and $\nohs\times\noh$}\lb{sec4}
In the previous section we related the Morse index of a differential operator on a multidimensional domain in $\R^n$ and the Maslov index of a path taking values in the set of Lagrangian subspaces of a rather sophisticated space of abstract boundary values. 
In this section, we show how to reformulate Theorems \ref{mormastper}, \ref{mormasper} and \ref{dir} 
using the Maslov index of yet another natural path taking values in the set of Lagrangian subspaces of the real symplectic Hilbert space $\mathfrak{H}=\nohs\times\noh$ of functions on the actual boundary of the domain. We collected some preliminary information used in this section in Appendix \ref{appA}.

\begin{hypothesis}\lb{b1}
Let $n\in\bbN,\ n\geq 2$, and assume that $\Omega\subset\R^n$ is a  quasi-convex domain.
\end{hypothesis} 
We refer to  \cite[Definition 8.9]{GM10} for the definition of quasi-convex domains and recall from \cite[Section 8]{GM10} that $\Omega$ is quasi-convex provided each point $x\in\partial\Omega$ has 
a neighbourhood $\Omega_x$ in $\Omega$ which is almost convex and has some additional Fourier multiplier property 
$MH^{1/2}_{\delta}$, see \cite[Definition 8.1]{GM10}. For instance, convex domains or domains with the local exterior 
ball conditions are almost convex, square Dini domains are $MH^{1/2}_{\delta}$, and $C^{1,r}$ domains with $r>1/2$ are
quasi-convex.
Following \cite{GM10} we consider the Dirichlet trace operator $\hatt{\gamma}_D$ and the normalized Neumann trace $\tN$, 
see the definitions and notations introduced in Appendix A.
\begin{proposition}\lb{mainb}
Assume Hypothesis \ref{b1}. Then the operator 
\begin{align}
&T:\dom(-\Delta_{\max})\rightarrow \nohs\times \noh\lb{b4},\\
& Tu:=(\gd u, \tN u)\lb{4.2}
\end{align}
is bounded when the space 
$\{u\in L^2(\Omega)| \Delta u\in L^2(\Omega)\}=\dom(-\Delta_{\max})$ is equipped with the natural graph norm 
$\|u\|_{gr(-\Delta)}:=\|u\|_{L^{2}(\Om)}+\|\Delta u\|_{L^{2}(\Om)}$. Moreover, this operator is onto, and
\begin{equation}\lb{b6}
\ker(T)=H^2_0(\Omega)=\{u\in L_2(\Omega): \Delta u\in L_2(\Omega),\  \gd u=0,\  \gn u=0\}. 
\end{equation}
\end{proposition}
\begin{proof}
The fact that $T$ is bounded follows from the definitions of the trace maps $\gd$ and $\tN$ 
(see Appendix A, Lemma \ref{wd}, Theorem \ref{nn}). 
Invoking Theorem \ref{m=m*}, one establishes the second equality in \eqref{b6}, 
to show the first equality we recall \eqref{kernel}, that is 
\begin{equation}
\lb{b7}
\ker(\tN)=H_0^2(\Omega)\dot{+}\{u\in L^2(\Om), \,\,-\Delta u=0 \}.
\end{equation}
Also, since $0\notin\spec(-\Delta_{D,\Omega})$, one has 
\begin{equation}
\lb{b77}
\ker(\gd)\cap\{u\in L^2(\Om), \,\,-\Delta u=0 \}=\{0\},
\end{equation}
and moreover,
\begin{equation}
\lb{b772}
H^2_0(\Omega)\subset\ker(\gd).
\end{equation}
Then $\ker(T)=\ker(\tN)\cap\ker(\gd)$ and \eqref{b7}, \eqref{b77}, \eqref{b772} imply
$\ker(T)=H_0^2(\Omega)$.
Next, we prove that $T$ is onto. Fix a vector $(g,f)\in \nohs\times\noh$. 
Since $\tN$ is surjective by Theorem \ref{nn} (moreover, \eqref{onto} holds) there exists 
$u_0\in H^2(\Omega)\cap H_0^1(\Omega)$, such that $\tN u_0=f$. By \cite[Theorem 10.4]{GM10}, the boundary value problem 
\begin{equation}
    -\Delta u =0,\,
    u\in L^2(\Om),\quad     
	\gd u =g   \text{\ on\ } \partial\Omega, \end{equation}
has a unique solution $v_0$. Employing $v_0\in \ker(\tN)$, by formula \eqref{b7}  one obtains 
\begin{align}
T (u_0+v_0)=(\tN(u_0+v_0),\gd(u_0+v_0))=(\tN(u_0),\gd(v_0))=(f,g)
\end{align}
since $\gd (u_0)=\gaD(u_0)=0$ by formula \eqref{comp}.
\end{proof}

\begin{definition} Assume Hypothesis \ref{b1} and introduce the symplectic Hilbert space $\mathfrak{H}$
equipped with the standard inner product and the symplectic form 
\begin{align}\lb{frakH}
\omega_{\mathfrak{H}}&:\mathfrak{H}\times \mathfrak{H} \rightarrow \R,\quad \mathfrak{H}:=\nohs\times\noh,\\
\omega_{\mathfrak{H}}((u_1,v_1),(u_2,v_2))&:={\lnoh v_2, u_1  \rnohs}-\lnoh v_1,u_2 \rnohs\no.
\end{align}
\end{definition}
\begin{definition} Assume Hypothesis \ref{b1} and introduce the symplectic Hilbert space $\cH_{\Delta}$
	equipped with the standard inner product and the symplectic form 
	\begin{align}
	\omega_{\cH_{\Delta}}&:{\cH_{\Delta}}\times {\cH_{\Delta}} \rightarrow \R,\quad \cH_{\Delta}:=\dom(-\Delta_{\max})/\dom(-\Delta_{\min}),\no\\
	\omega_{\cH_{\Delta}}([u],[v])&:=(-\Delta_{\max} u, v)_{L^2(\Omega)}-(u, -\Delta_{\max} v)_{L^2(\Omega)}\lb{frakH2}.
	\end{align}
\end{definition}
Next, using the trace map $T$ from \eqref{b4}, \eqref{4.2}, we introduce a map $T_\mathfrak{H}$ as follows
\begin{align}
T_\mathfrak{H}:\dom(T_\mathfrak{H})=\cH_{\Delta}\to\mathfrak{H}, \,\,\,T_\mathfrak{H}([u])=T(u).
\end{align}

\begin{theorem}\lb{symmor}
Assume Hypothesis \ref{b1}. Then the linear operator $T_\mathfrak{H}$
\begin{align}\lb{T_B}
& T_\mathfrak{H}: \cH_{\Delta}\rightarrow \mathfrak{H},\ \ \ T_\mathfrak{H}[u]:=Tu,
\end{align} 
is a bounded isomorphism, and thus boundedly invertible; moreover 
\begin{equation} \lb{c2}
\omega_{\mathfrak{H}}(T_\mathfrak{H}[u],T_\mathfrak{H}[v])=\omega_{\cH_{\Delta}}([u],[v])
\end{equation}
for all $u, v \in \cH_{\Delta}$.
\end{theorem}
\begin{proof}
The operator $T_\mathfrak{H}$ is well-defined, since by Theorem \ref{m=m*} for all  $u_1,u_2 \in [u]$, one has $u_1-u_2\in\dom(-\Delta_{\min})=H^2_0(\Omega)$, 
thus $T_\mathfrak{H}(u_1-u_2)=0$. 
Let $[u]\in\cH_{\Delta}$ and choose a representative $u_1\in[u]$, such that $\|[u]\|_{\cH_{\Delta}}=\|u_1\|_{gr(-\Delta)}$. Then
\begin{equation}
\|T_\mathfrak{H}[u]\|_{\mathfrak{H}}=\|Tu\|_{\mathfrak{H}}=\|Tu_1\|_{\mathfrak{H}}\leq C\|u_1\|_{gr(-\Delta)}=C\|[u]\|_{\cH_{\Delta}},
\end{equation}
where $C>0$ is the norm of operator $T$ from Proposition \ref{mainb}. The operator $T_\mathfrak{H}$ is bijective, since $T$ is onto and  $\ker (T)=H^2_0(\Om)=\dom(-\Delta_{\min})$ by Proposition  \ref{mainb}. The open mapping theorem guarantees that $T_\mathfrak{H}^{-1}$ is bounded. Finally, \eqref{c2} follows from \eqref{frakH}, \eqref{frakH2} and Green's formula \eqref{Green}.
\end{proof}
We are now ready to formulate a version of Theorem \ref{absfur} with the symplectic structure being defined on the space of functions on the boundary of the domain $\Omega$. We continue to denote the natural projection by $\gamma:\dom(-\Delta_{\max})\rightarrow \dom(-\Delta_{\max})/\dom(-\Delta_{\min}),$ $\gamma(u):=[u]$, and use $T$ for the trace operator from Proposition \ref{mainb}.
 \begin{theorem}\lb{lagrsets} 
Assume Hypothesis \ref{b1} and that $t\to V_t$ is in $C^1(\Sigma, L^{\infty}(\Omega))$. Let $-\Delta_{\cD}$ be a self-adjoint extension of $-\Delta_{\min}$ with some domain $\cD\subset L^2(\Omega,d^nx)$, and assume that $-\Delta_{\cD}$ has compact resolvent. Then

(i) $T({\cD})$ is a Lagrangian subspace in $\mathfrak{H}$,

(ii) $T(K_{\lambda,t}) \in F\Lambda (T(\cD))$ for all $(\lambda,t)\in \R\times [0,1]$, 
where $K_{\lambda,t}:=\ker(-\Delta_{\max}+V_t-\lambda)$,

(iii) the map 
 $\Upsilon_{\Delta}$ defined by $s\mapsto T\big(\ker(-\Delta_{\max}+V_{t(s)}-\lambda(s)\big)$ belongs to 
 $C^1(\Sigma_j, F\Lambda(T(\cD)))$, where we use parametrization \eqref{par2}--\eqref{par5}. 

(iv) Also, if $V(\cdot)\in C(\Sigma, L^{\infty}(\Omega))$, then so is $\Upsilon_{\Delta}$. 
\end{theorem}

\begin{proof}
Invoking Theorem \ref{bucp},  one applies Theorem \ref{absfur} with $A=-\Delta_{\min}$, $A_\cD=-\Delta_{\cD}$, 
obtaining $\gamma(K_{\lambda,t}) \in F\Lambda(\gamma(\cD))$. To show {\it(i)} and {\it (ii)} we observe that 
$T(\cD)=T_\mathfrak{H}(\gamma(\cD))$, $T(K_{\lambda,t})=T_\mathfrak{H}(\gamma(K_{\lambda,t}))$, and recall that $T_\mathfrak{H}$ is a symplectomorphism.  

To prove {\it(iii)}, we need to construct a smooth family of orthogonal projections $P^{\mathfrak{H}}_s\in\cB(\mathfrak{H})$ onto 
$\ran(P^{\mathfrak{H}}_s)=T\big(\ker(-\Delta_{\max}+V_{t(s)}-\lambda(s)\big)$. By Theorem \ref{absfur} 
there exists a family of smooth orthogonal projections $P_s\in\cB(\cH_{\Delta})$ onto $\ran(P_s)=\gamma\big(\ker(-\Delta_{\max}+V_{t(s)}-\lambda(s))\big)$. 
We shall show that the family of projections $\tilde P^{\mathfrak{H}}_s:=T_\mathfrak{H}\circ P_s\circ T_\mathfrak{H}^{-1}\in\cB(\mathfrak{H})$ (not necessarily orthogonal) 
has desired property. Let $u\in\ran(\tilde P^{\mathfrak{H}}_s)$, then 
$\tilde P^{\mathfrak{H}}_su=u$ or $T_\mathfrak{H}P_sT_\mathfrak{H}^{-1}u=u$ that is $P_sT_\mathfrak{H}^{-1}u=T_\mathfrak{H}^{-1}u$, therefore, by \eqref{T_B}, 
$T_\mathfrak{H}^{-1}u\in\ran(P_s)=\gamma\big(\ker(-\Delta_{\max}+V_{t(s)}-\lambda(s))\big)$ implying 
$u\in T\big(\ker(-\Delta_{\max}+V_{t(s)}-\lambda(s))\big)\big)$. 
Similarly one shows the inclusion $\ran(P^{\mathfrak{H}}_s)\supset T\big(\ker(-\Delta_{\max}+V_{t(s)}-\lambda(s)\big)$. 
Hence, $\tilde P^{\mathfrak{H}}_s\in\cB(\mathfrak{H})$ is 
a smooth family of projections onto $T\big(\ker(-\Delta_{\max}+V_{t(s)}-\lambda(s))\big)$. 
Proceeding as in \cite[Remark 3.4]{CJLS} using the family of smooth projections $\tilde P^{\mathfrak{H}}_s$, we construct a family of smooth orthogonal projections $P^{\mathfrak{H}}_s\in\cB(\mathfrak{H})$ onto 
$\ran(P^{\mathfrak{H}}_s)=T\big(\ker(-\Delta_{\max}+V_{t(s)}-\lambda(s))\big)$.
{\it(iv)} similar to {\it(iii)}.
\end{proof}

We are now ready to establish a relation between the Morse index of the Shr\"{o}dinger operator with the quasi-periodic, Dirichlet and Neumann boundary conditions and the Maslov index of the paths with values in the set of the Lagrangian subspaces of ${\mathfrak{H}}$.  Next, we reformulate Theorems \ref{mormastper}{\it (i)}, \ref{mormasper}{\it (i)} and \ref{dir} in terms of the Lagrangian paths of the (weak) traces of the solutions from the real symplectic Hilbert space $\mathfrak{H}$. Denote $\cX^{\mathfrak{H}}_{\vec{\theta}}:=T(\dom(-\Delta_{\vec{\theta},\R}))$.
\begin{theorem}
Assume Hypothesis \ref{phyp} {\it(ii)} and that $\vec{\theta}\not=0$. Then the Morse index of the operator $ -\Delta_{\vec{\theta}}+V$ and the Maslov index of the path $\Upsilon_\mathfrak{H}:[\tau,1]\rightarrow F\Lambda\big(\cX^{\mathfrak{H}}_{\vec{\theta}}\big)$ defined by $t\mapsto T\big(\ker(-\Delta_{\max}+V_{t,\R})\big)$ for all $t\in[\tau,1]$ with $\tau\in(0,1]$ sufficiently small are related as follows:
\begin{equation*}
2\mo_{\C}\big(-\Delta_{\vec{\theta}}+V\big)=
\begin{cases} \mi \big(\Upsilon,\cX^\mathfrak{H}_{\vec{\theta}}\big) &\text{if $\vec{\theta}\not=0$},\\
-\mi \big(\Upsilon,\cX^\mathfrak{H}_0\big)+2\mo(V(0))& \text{if $\vec{\theta}=0$, $V(0)$ is invertible}. \end{cases} 
\end{equation*}
\end{theorem}

\begin{theorem}
	Let $\Omega$ be an open bounded star-shaped Lipschitz domain in $\R^n$ centered at zero, and assume that $V\in C^1(\R^n, \R^{m\times m})$ and  $V(x)=V(x)^{\top}$, $x\in\Omega$. Denote $\cX_D^\mathfrak{H}=T(\dom\big(-\Delta_{D,\Omega})\big)$. Then the Morse index of the Schr\"{o}dinger operator with the Dirichlet boundary conditions and, for $\tau\in(0,1]$ sufficiently small, the Maslov index of the path $\Upsilon_\mathfrak{H}:[\tau,1]\rightarrow F\Lambda(\cX_D^\mathfrak{H})$, defined by $t\mapsto T\big(\ker\big(-\Delta_{\max}+V_{t,\R}\big)\big)$ for all $t\in[\tau,1]$, are related as follows:
	\begin{equation}\no
	2\mo_{\C}\big(-\Delta_{D,\Omega}+V\big)=-{\mi \big(\Upsilon,\cX_D^{\mathfrak{H}}\big)}.
	\end{equation}
Assume, in addition, that the matrix $V(0)$ is invertible. Denote $\cX_N^\mathfrak{H}=T(\dom\big(-\Delta_{N,\Omega})\big)$. Then the Morse index of the Schr\"{o}dinger operator with the Neumann boundary conditions and, for  $\tau\in(0,1]$ sufficiently small, the Maslov index of the path $\Upsilon_\mathfrak{H}:[\tau,1]\rightarrow F\Lambda(\cX_N^\mathfrak{H})$, defined by $t\mapsto T\big(\ker\big(-\Delta_{\max}+V_{t,\R}\big)\big)$ for all $t\in[\tau,1]$, are related as follows:
\begin{equation}\no
2\mo_{\C}\big(-\Delta_{N,\Omega}+V\big)=-{\mi \big(\Upsilon,\cX_N^{\mathfrak{H}}\big)}+2\mo(V(0)).
\end{equation}
\end{theorem}

To conclude this section we would like to compare the Lagrangian subspaces used in \cite{CJLS} and the ones defined in Theorem \ref{lagrsets}. This comparison shows that although the results in \cite{CJLS} and in the current paper are in the same spirit, neither of them follow from the other.
For simplicity, we will assume that $\partial\Omega$ is $C^{1,r}$ with $r>1/2$, so that $N^{1/2}(\partial\Omega)=H^{1/2}(\partial\Omega)$ and $(N^{1/2}(\partial\Omega))^*=H^{-1/2}(\partial\Omega)$. 
Let $V_t\in L^{\infty}(\Om)$ and consider two sets, 
\begin{align*}
&\mathcal{K}_{\lambda,t}=\{u\in H^{1}(\Om)\,|\,-\Delta u+V_tu-\lambda u=0 \,\,\hbox{in}\,\,H^{-1}(\Om)\},\\
&K_{\lambda,t}=\ker(-\Delta_{\max}+V_t-\lambda)\\
&\quad\quad=\{u\in L^{2}(\Om)\,|\,-\Delta u+V_tu-\lambda u=0\,\, \,
\hbox{in sence of distributions}\}.
\end{align*}
\begin{lemma} \lb{4.7} Assume that $V_t\in L^{\infty}(\Om)$. Then $\mathcal{K}_{\lambda,t}\subset K_{\lambda,t}$.
\end{lemma}
\begin{proof}
 Let $u\in\mathcal{K}_{\lambda,t}$, i.e. $u\in H^1(\Om)$ and  $(\nabla u,\nabla v)_{L^2(\Om)}=((\lambda-V_t)u,v)_{L^2(\Om)}$ for any $v\in H^{1}_0(\Om)$. 
 In particular, $(\nabla u,\nabla \phi)_{L^2(\Om)}=((\lambda-V_t)u,\phi)_{L^2(\Om)}$ for any $\phi\in C^{\infty}_0(\Om)$. 
 Applying Green's formula, we arrive at
 \begin{equation}
  (u, -\Delta\phi)_{L^2(\Om)}=(\nabla u,\nabla \phi)_{L^2(\Om)}=((\lambda-V_t)u,\phi)_{L^2(\Om)}\,\,
  \hbox{for any}\,\,\phi\in C^{\infty}_0(\Om).
 \end{equation}
Hence, $u\in K_{\lambda,t}$.
\end{proof}

We are ready to compare the subspaces used in \cite{CJLS} and in 
Theorem \ref{lagrsets}. Indeed,
the subspaces from \cite{CJLS} are of the type tr$(\cK_{\lambda,t})$, where \[\dom(\tr)=\cD:=\{u\in H^1(\Omega)\,\big|\,\Delta u\in L^2(\Omega)\,\text{ in $H^{-1}(\Om)$}\}\] and the map $\tr$ is defined by
\begin{align}\label{dfnTr}
\tr\colon\cD\to H^{1/2}(\dOm)\times H^{-1/2}(\dOm),\quad \tr u=(\gaD  u,\tilde\gamma_N u).
\end{align}
Here, $\gamma_N$ is the weak Neumann trace defined in \eqref{weaknt}.
First, we notice that the operators $T$ and tr do not agree even on the intersection of their domains. Secondly, in general $T(\cK_{\lambda,t})$  is not a Lagrangian subspace in $\mathfrak{H}$. We shall show this in case $V_t=0$ and $\lambda=0$, in fact we are claiming that $T(\cK_{0,t})$ doesn't not obey maximality assumption. By Lemma \ref{4.7}, $T(\cK_{0,t})\subset T(K_{0,t})$, moreover $T(K_{0,t})$ is Lagrangian in $\mathfrak{H}$, thus statement is proved, provided the inclusion is strict. Recall that $\gamma_D$ and $\gd$ agree on $H^1(\Omega)$, thus $\gd(\cK_{0,t})=\gamma_D(\cK_{0,t})\subset H^{1/2}(\partial\Omega)$, on the other hand $\gd(K_{0,t})=H^{-1/2}(\partial\Omega)$, so that sets of the first coordinates of $T(K_{0,t})$ and $T(\cK_{0,t})$ are not equal, hence the inclusion is strict.
This concludes the proof of the fact that the subspaces in \cite{CJLS} and in Theorem \ref{lagrsets} are indeed different.


\appendix\section{Dirichlet and Neumann trace operators}\label{appA}

In this appendix we recall several definitions and facts about
 various types of Dirichlet and Neumann trace operators recently discussed in \cite{GM10} and \cite{GM08}.

\begin{hypothesis}\lb{6.1L}
	Let $n\in\mathbb{N},n\geq2$, and assume that $\Om\subset\mathbb{R}^n$ is a bounded Lipschitz domain.
\end{hypothesis}

First, we define the strong trace operators.
Let us introduce the boundary trace operator $\gaD^0$ (the Dirichlet
trace) by
\begin{equation}\lb{2.4}
\gaD^0\colon C^0(\ol{\Om})\to C^0(\dOm), \quad
\gaD  ^0 u = u|_\dOm .
\end{equation}
By the standard trace theorem, see, e.g., \cite[Proposition 4.4.5]{T11}, there exists a bounded, surjective Dirichlet
trace operator 
\begin{equation}
\gaD  \colon H^{s}(\Om)\to H^{s-1/2}(\dOm) \hookrightarrow
L^2(\dOm), \quad 1/2<s<3/2.
\lb{6.1}
\end{equation}
Furthermore, the map has a bounded right inverse, i.e., given any $f\in H^{s-1/2}(\dOm)$ there exists $u\in H^{s}(\Om)$ 
such that $\gaD   u=f$ and $\|u\|_{H^{s}(\Om)}\leq C\|f\|_{H^{s-1/2}(\dOm)}$.

Next, retaining Hypothesis \ref{6.1L}, we introduce the Neumann trace operator $\gaN$
\begin{equation}\lb{Nstrong}
\gaN=\nu\cdot\gaD\nabla  \colon H^{s+1}(\Om)\to L^2(\dOm), \quad 1/2<s<3/2,
\end{equation}
where $\nu$ denotes the outward pointing normal unit vector to $\dOm$.
Furthermore, one can introduce the weak Neumann trace operator $\tilde\gamma_N$ by
\begin{equation}\lb{weaknt}
\tilde\gamma_N:\{u\in H^{s+1/2}(\Om)\,|\,\Delta u\in H^{s_0}(\Om)\}\to H^{s-1}(\dOm),\,\,s\in(0,1),\,s_0>-1/2.
\end{equation}

Assuming Hypothesis \ref{6.1L}, we now introduce the space
\begin{equation}\label{dfnN12}
N^{1/2}(\dOm):=\{g\in L^{2}(\dOm)\,|\,g\nu_j\in H^{1/2}(\dOm),\,1\leq j\leq n\},
\end{equation}
where the $\nu_j$'s are the components of $\nu$. We equip this space with the natural norm
\begin{equation}
\|g\|_{N^{1/2}(\dOm)}:=\sum_{j=1}^n\|g\nu_j\|_{H^{1/2}(\dOm)},
\end{equation}
and note that $N^{1/2}(\partial \Omega)=H^{1/2}(\partial \Omega)$ provided $\Omega$ is a $C^{1,r}$ domain with $r>1/2$, \cite[Lemma 6.2]{GM10}
\begin{lemma}[\cite{GM10}, Lemma 6.3]
	Assume Hypothesis \ref{6.1L}. Then the Neumann trace operator $\gaN$ considered in the context
	\begin{equation}\lb{gaN0}
	\gaN:H^2(\Om)\cap H^1_0(\Om)\to N^{1/2}(\dOm),
	\end{equation}
	is well-defined, linear, bounded, onto, and with a linear, bounded right-inverse. 
	In addition, the null space of $\gaN$ in \eqref{gaN0}
	is $H^2_0(\Om)$, the closure of $C^{\infty}_0(\Om)$ in $H^2(\Om)$.
\end{lemma}

\begin{lemma}[\cite{GM10}, Theorem 6.4]\lb{wd}
	Assume Hypothesis \ref{6.1L}. Then there exists a unique linear, bounded operator 
	\begin{equation}
	\gd:\{u\in L^{2}(\Om)\,|\,\Delta u\in L^{2}(\Om)\}\to (N^{1/2}(\dOm))^*,
	\end{equation}
	which is compatable with the Dirichlet
	trace, introduced in \eqref{6.1} and further extended in \cite[Lemma 3.1]{GM10}, in the sense that for 
	each $s\geq1/2$ one has 
	\begin{equation}\lb{comp}
	\gd u=\gaD u\,\,\,\hbox{for every}\, u\in H^{s}(\Om)\,\hbox{with}\,\,\Delta u\in L^{2}(\Om).
	\end{equation}
	Furthermore, this extension of the Dirichlet trace operator has dense range and allows for the following integration by 
	parts formula
	\begin{equation}
	\lnoh \gaN w, \gd u\rnohs= (\Delta w,u)_{L^{2}(\Om)}-(w,\Delta u)_{L^{2}(\Om)}, 
	\end{equation}
	valid for every $u\in L^{2}(\Om)$ with $\Delta u\in L^{2}(\Om)$ and every $w\in H^{2}(\Om)\cap H^{1}_0(\Om)$.
\end{lemma}
Assuming
Hypothesis \ref{6.1L} we define the space
\begin{equation}
N^{3/2}(\dOm):=\{g\in H^{1}(\dOm)\,|\,\nabla_{\tan}g\in (H^{1/2}(\dOm))^n\},
\end{equation}
equipped with the natural norm
$\|g\|_{N^{3/2}(\dOm)}:=\|g\|_{L^{2}(\dOm)}+\|\nabla_{\tan}g\|_{H^{1/2}(\dOm)^n}$.
Here, the tangential gradient operator $\nabla_{\text{tan}}: H^1(\partial \Omega)\mapsto L^2(\partial \Omega)^n$ is defined as $$f\mapsto \Big(\sum\limits_{k=1}^{n}\nu_k\frac{\partial f}{\partial \tau_{k,l}}\Big)_{l=1}^n,$$
and $\frac{\partial}{\partial \tau_{k,l}}$  is the tangential derivative, the bounded operator between $H^s(\partial \Omega)$ and $H^{s-1}(\partial \Omega),\ 0\leq s\leq 1$, that extends the operator 
$$\frac{\partial}{\partial \tau_{k,l}}:\psi \mapsto \nu_k(\partial_{l}\psi)\big|_{\partial \Omega}-\nu_l(\partial_{k}\psi)\big|_{\partial \Omega},$$
originally defined for $C^1$ function $\psi$ in a neighbourhood of $\partial \Omega$. Also, $N^{3/2}(\partial \Omega)=H^{3/2}(\partial \Omega)$ provided $\Omega$ is a $C^{1,r}$ domain with $r>1/2$,  \cite[Lemma 6.8]{GM10}.
\begin{lemma}[\cite{GM10}, Lemma 6.9] Assume Hypothesis \ref{6.1L}. Then the Dirichlet trace operator $\gamma_D$ considered in the context
	\begin{equation}\lb{gaD0}
	\gaD:\{u\in H^2(\Om)\,|\,\gaN u=0\}\to N^{3/2}(\dOm),
	\end{equation}
	is well-defined, linear, bounded, onto, and with a linear, bounded right-inverse. 
	In addition, the null space of $\gaD$ in \eqref{gaD0}
	is  $H^2_0(\Om)$.
\end{lemma}

\begin{lemma}[\cite{GM10}, Theorem 6.10]
	Assume Hypothesis \ref{6.1L}. Then there exists a unique linear, bounded operator 
	\begin{equation}
	\gn:\{u\in L^{2}(\Om)\,|\,\Delta u\in L^{2}(\Om)\}\to (N^{3/2}(\dOm))^*,
	\end{equation}
	which is compatable with the Neumann
	trace, introduced in \eqref{Nstrong} and further extended in \cite[(3.14)]{GM10}, in the sense that for 
	each $s\geq3/2$ one has 
	\begin{equation}
	\gn u=\tilde\gamma_N u\,\,\,\hbox{for every}\, u\in H^{s}(\Om)\,\hbox{with}\,\,\Delta u\in L^{2}(\Om).
	\end{equation}
	Furthermore, this extension of the Neumann trace operator has dense range and allows for the following integration by 
	parts formula
	\begin{equation}
	\lnth \gaD w, \gn u\rnths= (w,\Delta u)_{L^{2}(\Om)}-(\Delta w,u)_{L^{2}(\Om)}, 
	\end{equation}
	valid for every $u\in L^{2}(\Om)$ with $\Delta u\in L^{2}(\Om)$ and every $w\in H^{2}(\Om)$ with $\gaN w=0$.
\end{lemma}

Assuming Hypothesis \ref{6.1L}, we define, as in \cite[Lemma 7.1.]{GM10},
\begin{align}
-\Delta_{\max}&:\dom(-\Delta_{\max})\subset L^2(\Omega)\rightarrow L^2(\Omega),\no\\
\dom(-\Delta_{\max})&=\big\{u\in L^{2}(\Om)\big|\  \Delta u\in L^{2}(\Om)\big\},\no\\
-\Delta_{\max}u&=-\Delta u,\ \  (\text{in the sence of distribudtions}).\no
\end{align}
\begin{theorem}[\cite{GM10}]
	Assume Hypothesis \ref{6.1L}. Then the maximal Laplacian $-\Delta_{\max}$ is a closed, densely defined operator for which
	\begin{align*}
	&H^2_0(\Om)\subseteq\dom((-\Delta_{\max})^*)=\big\{u\in L^{2}(\Om)|\  \Delta u\in L^{2}(\Om), \ \gd(u)=0,\ \gn(u)=0\big\}.
	\end{align*}
\end{theorem}
We also define the minimal operator by
\begin{align}
\dom(-\Delta_{\min})=H^2_0(\Om),\,-\Delta_{\min}u=-\Delta u.
\end{align}
\begin{theorem}[\cite{GM10}, Corollary 7.2]
	Assume Hypothesis \ref{6.1L}. Then the minimal Laplacian is a densely defined, symmetric operator which satisfies
	\begin{align}\lb{mm}
	-\Delta_{\min}\subseteq(-\Delta_{\max})^*,\,\,-\Delta_{\max}\subseteq(-\Delta_{\min})^*.
	\end{align}
\end{theorem}
\begin{hypothesis}\lb{b1111}
	Let $n\in\bbN,\ n\geq 2$, and assume that $\Omega\subset\R^n$ is a  quasi-convex domain (see \cite[Definition 8.9]{GM10}).
\end{hypothesis} 
\begin{theorem}[\cite{GM10}, Theorem 8.14]\lb{m=m*}
	Assume Hypothesis \ref{b1111}. Then
	\begin{align*}
	\dom(-\Delta_{\min})&=H^2_0(\Om)=\{u\in L^{2}(\Om)|\  \Delta u\in L^{2}(\Om) ,\ \gd(u)=0,\ \gn(u)=0\},\\
	-\Delta_{\min}&=(-\Delta_{\max})^*,\,\,-\Delta_{\max}=(-\Delta_{\min})^*.
	\end{align*}
\end{theorem}
A nice additional feature of the quasi-convex domains is that the trace operators are onto, see \cite[Theorem 10.2, 10.6]{GM10}.

Assuming Hypothesis \ref{b1}, we introduce the Dirichet-to-Neumann map $M_{D,N}$ associated with $-\Delta$ on $\Om$ as follows
\begin{equation}
M_{D,N}:
(N^{1/2}(\dOm))^*\to(N^{3/2}(\dOm))^*:
f\mapsto-\gn(u_D),  
\end{equation}
where $u_D$ is the unique solution of boundary value problem
\begin{equation}
-\Delta u=0\,\,\hbox{in}\,\,\Om,\,\,\,\,u\in L^{2}(\Om),\,\,\gd u=f\,\,\hbox{in}\,\,\dOm.
\end{equation}

Following \cite[Section 12]{GM10}, we consider the regularized Neumann trace operator on quasi-convex domains.
\begin{theorem}[\cite{GM10}, Theorem 12.1]\lb{nn}
	Assume Hypothesis \ref{b1}. Then the map
	\begin{align}
	\tN:\{u\in L^2(\Omega)| \Delta u\in L^2(\Omega)\}\rightarrow \noh,\,
	\tN u:=\gn u+ M_{D,N}(\gd u),
	\end{align}
	is well-defined, linear, and bounded when the space 
	$$\{u\in L^2(\Omega)| \Delta u\in L^2(\Omega)\}=\dom(-\Delta_{\max})$$ is equipped with the natural graph norm 
	$u\mapsto \|u\|_{L^{2}(\Om)}+\|\Delta u\|_{L^{2}(\Om)}$. Moreover, this operator is onto. In fact,
	\begin{equation}\lb{onto}
	\tN(H^2(\Om)\cap H^1_0(\Om))=\noh.
	\end{equation}
	Also, the null space of the map $\tN$ is given by
	\begin{equation}\lb{kernel}
	\ker(\tN)=H_0^2(\Omega)\dot{+}\{u\in L^2(\Om),\,\, -\Delta u=0 \}.
	\end{equation}
	Finally, the following Green formula holds for every $u,v\in\dom(-\Delta_{\max})$,
	\begin{align}\lb{Green}
	&(-\Delta u,v)_{L^{2}(\Om)}-(u,-\Delta v)_{L^{2}(\Om)}\no\\
	&=-\lnoh \tN u, \gd v \rnohs+{\lnoh \tN v, \gd u \rnohs}.
	\end{align}
	
\end{theorem}

The following result can be found, for example, in \cite[Proposition 2.5]{BB12}.
\begin{theorem}\lb{bucp}
Let $\Omega$ be an open bounded star-shaped Lipschitz domain in $\R^n$ and let $L=-\Delta+V$ with a potential $V\in L^{\infty} (\overline{\Omega})$. If $u\in H^1(\Omega)$ is a weak solution of the Schr\"{o}dinger equation $Lu=0$ that satisfies $\gamma_Du=0$ and $\gamma_Nu=0$ then $u=0$.
\end{theorem}

\begin{lemma}[\cite{CJLS}, Lemma 3.8]\lb{crossex}
	Let $\{P_s\}_{s\in\Sigma}$ be a family of orthogonal projections on a Hilbert space $\cH$ such that the function $s\mapsto P_s$ is in $C^1(\Sigma;\cB(\cH))$ for some $\Sigma=[a,b]\subseteq\R$. Then for any $s_0\in\Sigma$ there exists a neighborhood $\Sigma_0$ in $\Sigma$ containing $s_0$ and a family of operators $\{B_s\}$ from $\ran(P_{s_0})$ into $\ker(P_{s_0})$ such that 
	the function $s\mapsto B_s$ is in $C^1(\Sigma_0;\cB(\ran(P_{s_0}),\ker(P_{s_0})))$ and for all $s\in\Sigma_0$, using the decomposition 
	$\cH=\ran(P_{s_0})\oplus\ker(P_{s_0})$, we have
	\begin{equation}\lb{greq}
	\ran(P_s)=Graph(B_s)=\{q+B_sq: q\in\ran(P_{s_0})\}.
	\end{equation}
	Moreover,
	$B_s\to0$ in $\cB(\ran(P_{s_0}),\ker(P_{s_0}))$ as $s\to s_0$.
\end{lemma}


\end{document}